\documentclass[11pt,reqno]{amsart}

\usepackage{amsmath,amssymb,amsthm,amsfonts}
\usepackage{hhline}
\usepackage{marginnote}  %
\usepackage{accents} 
\usepackage{mathtools} %
\usepackage{enumerate}
\usepackage{hyperref}

\usepackage{multirow}
\usepackage{xcolor}
\allowdisplaybreaks
\mathtoolsset{showonlyrefs}

\DeclareMathOperator{\ff}{ff}

\DeclareMathOperator{\supp}{supp}

\DeclareMathOperator{\I}{I}
\DeclareMathOperator{\II}{II}

\DeclareMathOperator{\Sc}{sc}
\DeclareMathOperator{\bd}{b}
\DeclareMathOperator{\D}{d}

\newcommand{\hexp}{\widehat{\exp}}

\newcommand{\R}{\mathbb{R}}
\newcommand{\B}{\mathbb{B}}

\renewcommand{\S}{\mathbb{S}}

\newcommand{\tM}{\widetilde{M}}

\newcommand{\p}{\partial}
\newcommand{\n}{\nabla}
\newcommand{\oeta}{\overline{\eta}}
\newcommand{\on}{\overline{\n}}
\newcommand{\hn}{\widehat{\nabla}}

\newcommand{\oM}{\overline{M}}
\newcommand{\oN}{\overline{N}} 

\newcommand{\ta}{{\widetilde{\alpha}}}

\newcommand{\oX}{\overline{X}}

\newcommand{\x}{{x_\eta}}
\newcommand{\oU}{\overline{U}}

\newcommand{\oA}{\overline{A}}

\newcommand{\hI}{\widehat{I}}
\newcommand{\hA}{\widehat{A}} 
\newcommand{\oI}{\overline{I}} 

\newcommand{\og}{\overline{g}}

\newcommand{\hphi}{\widehat{\phi}}

\newcommand{\pM}{\partial M}

\newcommand{\calF}{\mathcal{F}}
\newcommand{\calG}{\mathcal{G}}

\newcommand{\calT}{\mathcal{T}}
\newcommand{\calU}{\mathcal{U}}
\newcommand{\calV}{\mathcal{V}}
\newcommand{\calW}{\mathcal{W}}

\newcommand{\cU}{\mathcal{U}}

\newcommand{\cA}{\mathcal{A}}
 
\newcommand{\cAt}{\widetilde{\mathcal{A}}}

\newcommand{\cI}{\mathcal{I}}

\let\epsilon\varepsilon
\let\phi\varphi

\renewcommand{\a}{\alpha}
\renewcommand{\b}{\beta}
\newcommand{\g}{\gamma}
\renewcommand{\d}{\delta}
\newcommand{\e}{\epsilon}

\newcommand{\G}{\Gamma}
\newcommand{\oG}{\overline{\Gamma}}
\newcommand{\hG}{\widehat{\Gamma}} 
\newcommand{\hg}{\widehat{\gamma}} 
\newcommand{\oga}{\overline{\gamma}}  
\renewcommand{\r}{\rho}
\renewcommand{\t}{\tau}
\renewcommand{\k}{\kappa}
\renewcommand{\l}{\lambda}
\newcommand{\s}{\sigma}
\renewcommand{\th}{\theta}

\newcommand{\w}{\omega}

\newcommand{\tU}{\widetilde{U}}
\newcommand{\hO}{\widehat{\Omega}}
\newcommand{\oO}{\overline{\Omega}} 

\newcommand{\ep}{\epsilon}
\newcommand{\ga}{\gamma}
\newcommand{\de}{\delta}

\newcommand{\al}{\alpha}
\newcommand{\be}{\beta}
\newcommand{\gb}{\overline{g}}

\newtheorem{theorem}{Theorem}

\newtheorem{definition}{Definition}
\numberwithin{definition}{section}

\newtheorem{proposition}[definition]{Proposition}

\newtheorem{lemma}[definition]{{Lemma}}

\newtheorem{remark}[definition]{Remark}

\newtheorem{corollary}[definition]{Corollary}

\newtheorem{theoremofothers}[definition]{{Theorem}}

\numberwithin{equation}{section}

\newtheorem{manualpropositioninner}{Proposition}
\newenvironment{manualproposition}[1]{%
  \renewcommand\themanualpropositioninner{#1}%
  \manualpropositioninner
}{\endmanualpropositioninner}

\let \o \undefined
\def \o#1{\overline{#1}}
\def\fr#1#2{\frac{#1}{#2}}

\let\td\undefined
\def \td#1{\widetilde{#1}}
\let\implies\Rightarrow

\title[Local X-ray Transform on AH Manifolds]{Local X-ray Transform on
  Asymptotically Hyperbolic Manifolds via Projective Compactification}

\author[Nikolas Eptaminitakis and C. Robin Graham]{Nikolas Eptaminitakis}
\address{Department of Mathematics, Purdue University,
West Lafayette, IN 47907, USA}
\email{neptamin@purdue.edu}

\author[]{C. Robin Graham}
\address{Department of Mathematics, University of Washington,
Box 354350\\
Seattle, WA 98195-4350, USA}
\email{robin@math.washington.edu}


\begin{document}

\begin{abstract}
We prove local injectivity near a boundary point for the geodesic X-ray
transform for an asymptotically hyperbolic metric even mod $O(\rho^5)$ in dimensions three and higher.      
\end{abstract}

\maketitle
\thispagestyle{empty}

\begin{center}
{\it Dedicated to the memory of Vaughan Jones}
\end{center}

\bigskip

\section{Introduction}
The problem of recovering a function $f$ from its geodesic X-ray transform
\begin{equation} 
If(\g)=\int_\g fds,
\end{equation} where $\g$ is a geodesic of a Riemannian metric $g$ on a
Riemannian manifold and $ds$ denotes integration with respect to $g$-arc
length, has been studied extensively since the early 20th century, starting
with the Radon transform in the 2-dimensional Euclidean space
(\cite{Radon17}).  Aside from its intrinsic geometric interest, this
question arises in numerous applications, including medical, geophysical
and ultrasound imaging; for a comprehensive recent survey see \cite{IlmavirtaMonard+2019+43+114}. A major breakthrough in the study of the geodesic
X-ray transform was the proof by Uhlmann-Vasy 
(\cite{Uhlmann2016}) of local injectivity near a boundary point on manifolds 
of dimension at least 3 with strictly convex boundary.  In this paper we 
prove an analog of the Uhlmann-Vasy result on asymptotically hyperbolic
manifolds.

Let $(\oM^{n+1},\p\oM)$ be a compact manifold with boundary and $M$ be  
its interior.  
A $C^\infty$ metric $g$ on $M$ is called asymptotically hyperbolic (AH) if
for some (and hence any) smooth boundary defining function $\r$ (that is,
$\r\big|_{\p\oM}=0$, $\r>0$ on $M$, $d\r\big|_{\p\oM}\neq 0$) the 
Riemannian metric $\overline{g}:=\r^2 g$ on $M$ extends to a smooth metric
on $\oM$ with the additional property that
$\displaystyle|d\r|^2_{\og}\equiv 1$ on $\p\oM$.  We denote by  
$h=\overline{g}|_{T\p\oM}$ the induced metric on $\p\oM$.   
As shown in \cite{MR2941112}, $(M,g)$ is a complete Riemannian manifold
with sectional curvatures approaching $-1$ as $\r\to 0$. 
The classical example of an AH manifold is the Poincar{\'e} ball model of 
the hyperbolic space of constant sectional curvature $-1$, the manifold being
the Euclidean unit ball $\mathbb{B}^{n+1}=\{(x^1,\dots,x^{n+1})\in
\R^{n+1}:|x|<1\}$ with the metric  
\[
g_H:=4\dfrac{\sum_{j=1}^{n+1}(dx^j)^2} { (1-|x|^2)^2}.
\]
Interest in the study of AH manifolds has risen in the past 20 years since  
the AdS/CFT conjecture, proposed in \cite{MR1633016}, related conformal
field theories with gravity theories on AH spaces.  

Since a boundary
 defining function $\rho$ 
is determined only up to a smooth
positive multiple, $g$ determines a conformal family $\mathfrak{c}$ of 
metrics on the boundary given by $\mathfrak{c}=[h]$.  This conformal  
class of metrics is called the \textit{conformal infinity} of $g$.    
In \cite{MR1112625}, Graham and Lee show that for each conformal
representative $h\in \mathfrak{c}$ there exists a unique
boundary defining function $\r$ inducing a product decomposition
$[0,\e)_\r\times \p\oM$ of a collar neighborhood of the boundary such that the
metric can be written in the form 
\begin{equation}\label{eq:normal_form}
g=\fr{d\r^2+h_\r}{\r^2},
\end{equation}
where $h_\r$ is a one-parameter family of metrics on $\p\oM$, smooth
in $\r$ up to $\r=0$, with $h_0=h$. We say that an 
AH metric is in \textit{normal form} if it is written as in
\eqref{eq:normal_form}.  
Note that \eqref{eq:normal_form} implies that  the equality
$|d\r|^2_{\og}=1$ 
is valid in a neighborhood of $\p\oM$ instead of just on $\p\oM$. 
In this paper we will be concerned with AH metrics that are 
\textit{even mod $O(\rho^N)$}, where $N$ is a positive odd integer.
This means that whenever $g$ is
written in normal form \eqref{eq:normal_form} in a neighborhood of $\p\oM$, 
one has 
\begin{equation}\label{eq:even_to_order_N} 
  (\p_\r)^mh_\r\big|_{\r=0}=0 \text{ for } m \text{ odd, }
1\leq m< N. 
\end{equation}
In the case when \eqref{eq:even_to_order_N} holds for any odd 
$N> 0$ the metric $g$ will be called \textit{even}. 
\noindent As shown in \cite[Lemma 2.1]{MR2153454}, evenness mod 
$O(\rho^N)$ is a well defined property of an AH metric, independent of the
chosen conformal representative determining the normal form 
\eqref{eq:normal_form}.  

A unit-speed geodesic $\g$ for $g$ is said to be trapped if either 
$\liminf_{t\to\infty}\r(\g(t))>0$ or
$\liminf_{t\to-\infty}\r(\g(t))>0$.  If $\ga$ is not  
trapped, then $\lim_{t\to\pm\infty}\ga(t)\in\p\oM$ exists and
$\r(\ga(t))=O(e^{-|t|})$.  (See \cite{MR2941112} or
\cite[Lemma 2.3]{2017arXiv170905053G}.)  In this case we define 
\begin{equation}\label{eq:ray_transform}
If(\ga):=\int_{-\infty}^\infty f(\g(t))\,dt
\end{equation}
for $f$ such that the integral converges.  

Injectivity of the X-ray transform has been studied in various settings
overlapping with AH spaces.  Classical results on 
hyperbolic space viewed as a symmetric space can be found in
\cite{MR2743116}.  
More recently, \cite{2016arXiv161204800L} and 
\cite{2017arXiv170510126L} consider injectivity of the X-ray transform in
the setting of Cartan-Hadamard manifolds, which are by definition complete,
simply connected manifolds of non-positive curvature; the underlying
manifolds are diffeomorphic to $\R^n$.  Injectivity results  
specifically in the setting of AH manifolds can be found in
\cite{2017arXiv170905053G}. 

We will focus on \eqref{eq:ray_transform} restricted to a subset of 
geodesics.  
If $U\subset \oM$ (typically an open neighborhood of a point $p\in \p\oM$,
or its closure), a geodesic is said to be $U$-local if  
$\g(t)\in U$ for all $t\in\R$ and 
$\lim_{t\to \pm \infty}\g(t)\in U\cap \p\oM$.  The set $\Omega_U$ of
$U$-local geodesics is  nonempty if $U$ is any open neighborhood 
of a boundary point; this is a consequence of the existence of ``short'' 
geodesics (see \textsection 2.2 of \cite{2017arXiv170905053G}).

As we will indicate in Section~\ref{sec:proof_in_even_case}, for $U$ a
small neighborhood of a boundary point, the map 
$f\to If|_{\Omega_U}$ can be defined on $\r^{3/2}L^2(U;dv_{\og})$ with
values in an appropriate $L^2$ space (here $dv_{\o{g}}$ denotes the volume
form induced by the smooth metric $\o{g}$ on $\oM$).   

\begin{theorem}\label{thm:main}
Let $\oM$ be a manifold with boundary of dimension at least $3$, with its interior
endowed with an asymptotically hyperbolic metric $g$ that is even mod
$O(\rho^5)$.  
Given any neighborhood $O$ in $\oM$ of $p\in \p\oM$, there exists a 
neighborhood $U\subset O$ in $\oM$ of $p$ such that 
$f\to If\big|_{\Omega_U}$  is injective on $\r^{3/2}L^2(U;dv_{\og})$.     
\end{theorem}

\noindent
We expect that local injectivity holds for general asymptotically
hyperbolic metrics, but that other techniques will be needed to prove
this. 
Likewise, we expect that the hypothesis $f \in \rho^{3/2} L^2(U;dv_{\o{g}})$ can be weakened.

Our approach is motivated by the following observation.  Recall that 
the Klein model for hyperbolic space is another metric on $\B^{n+1}$
obtained from the Poincar\'e metric by a change of the radial variable.  
Geodesics for the Klein model are straight line segments in $\R^{n+1}$
under suitable parametrizations.  So the hyperbolic X-ray transform can
be identified with the Euclidean X-ray transform applied to a function 
supported in the unit ball, modulo changing the parameter of integration on
each geodesic. 
This observation has been utilized in the study of the hyperbolic Radon
transform, see e.g. \cite{MR1242891}.  There is an analogous relation for
even AH metrics.  An  
even AH metric induces what we call an {\it even structure} on
$(\oM,\p\oM)$ subordinate to its smooth structure. This is a subatlas of
the atlas 
defining the smooth structure, with the property that all the transition  
maps for the even structure are even diffeomorphisms.  
One can use the even
structure to define a new smooth structure $(\oM_e,\p\oM_e)$ on the
topological manifold with boundary underlying $(\oM,\p\oM)$ by introducing
$r=\r^2$ as a new boundary defining function.  As outlined at the end of
\textsection 4 of 
\cite{MR2858236}, when viewed relative to the smooth 
structure $(\oM_e,\p\oM_e)$, the metric $g$ is projectively 
compact in the sense that its Levi-Civita connection is projectively
equivalent to a connection $\hn$ smooth up to the boundary, i.e. its
geodesics agree up to parametrization with the geodesics of $\hn$.  The
connection 
$\hn$ need not be the Levi-Civita connection of a metric as happens 
on hyperbolic space, but the Uhlmann-Vasy local injectivity result applies
also to the X-ray transform for smooth connections, so local injectivity
for even AH metrics follows just by quoting \cite{Uhlmann2016}.

The introduction of $r=\rho^2$ as a new defining function to pass from $\oM$ to $\oM_e$ is a key step in Vasy's approach to microlocal analysis on even AH manifolds; see \cite{MR3117526,MR3135765, MR3581325}.

If the AH metric $g$ is not even, one can still introduce an even structure
and a corresponding $(\oM_e,\p\oM_e)$ by introducing $r=\r^2$ as a new
boundary defining function.  But in this case the connection $\hn$ is no
longer smooth up to the 
boundary: its Christoffel symbols have expansions in $\sqrt{r}$.  If 
$\p_\r h_\r\big|_{\r=0}\neq 0$ in \eqref{eq:normal_form}, 
then the Christoffel symbols have $r^{-1/2}$
terms so $\hn$ is not even continuous up to the boundary.  If 
$\p_\r h_\r\big|_{\r=0}= 0$ but $(\p_\r)^3h_\r\big|_{\r=0}\neq 0$, then
$\hn$ has $\sqrt{r}$ terms so it is continuous but not Lipschitz.  Our 
assumption that $g$ is even modulo $O(\r^5)$ guarantees that $\hn$ is at
least a $C^1$ connection.  

In principle one could try to extend directly the proof in  
\cite{Uhlmann2016} to the case of a $C^1$ connection like $\hn$.  But the
microlocal methods do not seem very well suited to such an analysis.
Instead we argue by perturbation:  $\hn$ is a perturbation of a smooth
connection, and the perturbation gets smaller the closer one gets to the
boundary.  For the quantitative control needed to carry this out, we 
need to use not only the local injectivity result of \cite{Uhlmann2016},
but also the associated stability estimate.  We briefly indicate how this 
goes, beginning by describing this stability estimate.

Let $\on$ be a smooth connection on a manifold with strictly convex
boundary $(\oM_e,\p \oM_e)$, of dimension at least 3, let $r$ be a boundary
defining function, and $\tM$ a closed manifold of the same dimension containing $\oM_e$.   
The authors of \cite{Uhlmann2016} constructed a one-parameter family of  ``artificial
boundaries'' near a point $p\in \p\oM_e$ given by $\hat{x}=-\eta$, where 
$\hat{x}\in C^\infty(\tM)$ satisfies $\hat{x}(p)=0$ and $d\hat{x}(p)=-dr(p)$, and $\eta>0$,
and showed injectivity of the 
X-ray transform $\oI$ of $\on$ restricted to geodesics in $\oM_e$ entirely
contained in $U_\eta:=\{\hat{x}\geq-\eta\}\cap\{r\geq 0\}$, for $\eta$
sufficiently small  
(see Figure \ref{figurename_artificial_boundary}).  
\begin{figure}[ht]
\begin{center}
	\includegraphics[scale=.4]{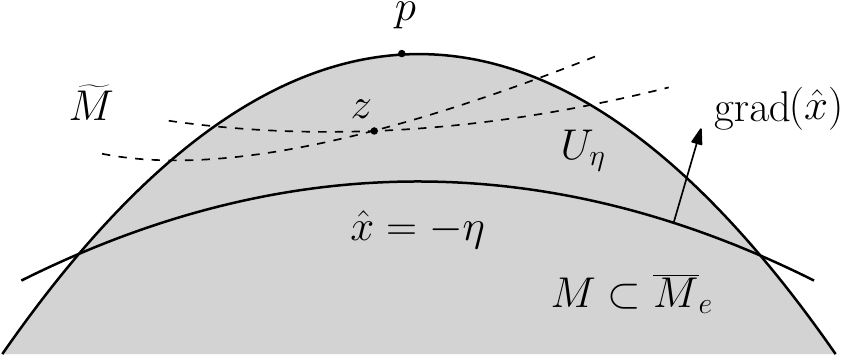}
	\caption{The artificial boundary.}
	\label{figurename_artificial_boundary}
	
\end{center}
	
\end{figure}
The proof is based on the construction of a   
family of ``microlocalized normal operators'' $\oA_{\chi,\eta,\s}$ each one of which is,
roughly speaking, the   
conjugate by exponential weights of the average of $\oI f$ over the 
set of such geodesics passing through a given point.  Here $\s$ is the parameter in the exponential 
weight and $\chi$ is a cutoff function.  
They showed that for appropriately chosen $\chi$, the operator
$\oA_{\chi,\eta,\s}$ is an elliptic 
pseudodifferential operator in Melrose's scattering calculus 
which for sufficiently small $\eta$ has trivial kernel
when acting on functions supported in $U_\eta$, and derived the stability
estimate 
\begin{equation}\label{stability}
\|f\|_{L^2(U_\eta)}\leq
C\|\oA_{\chi,\eta,\s}f\|_{H_{\Sc}^{1,0}(O_\eta)},
\end{equation}
where $H_{\Sc}^{1,0}$ denotes a scattering Sobolev space and $O_\eta$ is a
neighborhood of $U_\eta $ in $\oX_\eta:=\{\hat{x}\geq-\eta\}$.

If $g$ is an AH metric even mod $O(\r^N)$, its
Levi-Civita connection is projectively equivalent as described above to a
connection $\hn$ of the  
form $\hn=\on+ r^{N/2-1}B$ on $\oM_{e}$, where $\on$ and $B$ are smooth.   
If $N\geq 5$, then $\hn$ is $C^1$, so the constructions of its X-ray
transform $\hI$ and the operator $\hA_{\chi,\eta,\s}$ can be carried out
just as for the smooth connection $\on$.  We show that the norm of the
perturbation operator
\begin{equation}\label{po}
\hA_{\chi,\eta,\s}-\oA_{\chi,\eta,\s}: 
L^2(U_\eta)\to H_{\Sc}^{1,0}(O_\eta) 
\end{equation}
goes to zero as $\eta \to 0$.  
This gives an estimate of the form
\eqref{stability} for $\hA_{\chi,\eta,\s}$ for $\eta $ sufficiently
small, which implies local injectivity since $\hA_{\chi,\eta,\s}$ factors
through the X-ray transform $\hI$.  
The perturbation operator is estimated as in the classical Schur criterion  
bounding an $L^2$ operator norm by the sup of the $L^1$ norms of the
Schwartz kernel in each variable separately.  We lift the kernels of the
operators $\hA_{\chi,\eta,\sigma}$ and $\oA_{\chi,\eta,\sigma}$ to a
blown up space which is a refinement of Melrose's double stretched space (see
\cite{MR1291640}), where their singularities are more easily analyzed. 
Due to the fact that the connection $\hn$ is only of class $C^1$, some
rather technical  analysis is required near each boundary face and corner
of the blow-up to conclude that the kernel of $\hA_{\chi,\eta,\s}$ is
sufficiently regular that the norm of the perturbation operator vanishes in
the limit as $\eta\to 0$.

As in \cite{Uhlmann2016}, the method of proof naturally 
yields reconstruction via a Neumann series and a stability estimate for 
$\hI$ acting between Sobolev spaces on $\oM_e$ and the sphere bundle
$S^0\oM_e$ of a smooth metric $g^0$ on $\oM_e$, which we use to parametrize   
geodesics for $\hn$.  
One could pull back such an estimate to obtain one
for $I$ acting between Sobolev spaces on $M$ and the sphere bundle of $g$,
but we will not pursue this here.      
Moreover, one could obtain a global injectivity result in the same way as
in \cite{Uhlmann2016} provided the compact manifold with boundary $\{\r\geq
\e\}$ admits a strictly convex foliation, for all $\e$ sufficiently small. 
We mention that Vasy recently used semiclassical analysis to provide a
simplified, compared to \cite{Uhlmann2016}, proof of injectivity of the
global and local X-ray transform on compact manifolds with boundary
admitting a convex foliation (\cite{vasy2020semiclassical}). 
Global injectivity is shown there without the need for localization and layer stripping.
In the present work, working with the local transform is essential for the
aforementioned perturbation argument showing Theorem \ref{thm:main} for AH
metrics which are even mod $O(\r^5)$, and we follow the original
formulation of \cite{Uhlmann2016}.

In Section~\ref{sec:even_structure_definition} we define even structures on
a manifold with boundary and construct the new manifold with boundary
$(\oM_e,\p\oM_e)$ obtained by introducing $r=\r^2$ as a new defining
function.  We show that via this construction, even asymptotically 
hyperbolic metrics are the same as projectively compact metrics, only 
viewed relative to different smooth structures near infinity.  In 
Section~\ref{sec:proof_in_even_case} we use this observation to relate the 
X-ray transforms for $g$ and $\hn$, and then deduce Theorem~\ref{thm:main}  
for even AH metrics.  Section~\ref{nonsmooth} begins the analysis for the 
$C^1$ connection $\hn$ arising from an AH metric even mod $O(\r^5)$.  We
decompose $\hn$ into a smooth projectively compact connection $\on$ 
and a nonsmooth error term and extend both to the larger manifold $\tM$ in
such a way that they agree outside of $M_e$. 
We also prove Lemma~\ref{lm:regularity}, which  
states that the exponential map for $\hn$ has one more degree of regularity
than expected.  In Section~\ref{estimates} we review scattering Sobolev 
spaces on a manifold with boundary, the construction of the 
microlocal normal operator $A_{\chi,\eta,\s}$, and the stability estimate
\eqref{stability}.
We also  show how Theorem~\ref{thm:main} follows from
Proposition~\ref{prop:mapping_diff}, which is the assertion that the norm
of the perturbation operator \eqref{po} goes to zero as $\eta\to 0$.  In
Section~\ref{kernelanalysis} we describe the blown-up double space, 
analyze in detail the lift of the kernel of $A_{\chi,\eta,\s}$ to this
space, and conclude with the proof of Proposition~\ref{prop:mapping_diff}.

Throughout this paper and unless otherwise stated, given an
$n+1$-dimensional manifold with boundary (such as $(\oM,\p \oM)$ or
$(\oM_e,\p \oM_e)$), lower case Latin indices $i$, $j$, $k$ label  
objects on the manifold and run between $0$ and $n$ in coordinates. 
Lower case Greek indices $\al$, $\be$, $\ga$ label objects on 
the boundary and run between $1$ and $n$ in coordinates.  So a Latin
index corresponds to a pair $i\leftrightarrow (0,\al)$.

\section{\texorpdfstring{Even Asymptotically Hyperbolic $=$ Projectively
    Compact}{{Even Asymptotically Hyperbolic equals Projectively Compact}}}  
\label{sec:even_structure_definition}

This paper is based on an equivalence between even asymptotically
hyperbolic metrics and projectively compact metrics, briefly outlined at
the end of Section 4 of \cite{MR2858236}.  Since it is central to the
paper, we 
describe this equivalence in more detail.  We begin by recalling the 
notions of projective equivalence and projectively compact metrics.  A
reference for projective equivalence is 
\cite[\textsection 5.24]{poor1981differential}.   

Two torsion-free connections $\n$ and $\hn$ on a smooth manifold are
said to be 
projectively equivalent if they have the same geodesics up to
parametrization.  This is equivalent to the condition that their difference
tensor $\hn-\n$ is of the form  
$v_{\!(i}^{\phantom{k}} \delta_{j)}^k=\frac12 (v_i \delta_j^k+v_j \delta_i^k)$ for some
1-form $v$.  If $\ga(t)$ is a geodesic for $\n$, then $\ga(t(\tau))$ is a
geodesic for $\hn$, where $t(\t)$ solves the differential equation
$t''=\mu(t)(t')^2$ with $\mu(t)=-v_{\ga(t)}({\ga}'(t))$.  If $v=du$
happens to be exact, then this equation for the parametrization reduces to
the first order equation 
\begin{equation}\label{parameterchange}
t'=ce^{-u(\ga(t))} 
\end{equation}
which can be integrated by separation of variables.

Let ${}^eg$ be a metric on the interior of a manifold with boundary
$(\oM_e,\p\oM_e)$.  (The explanation for the super/subscript $e$ will be  
apparent shortly.  For now this is just an inconsequential notation.)     
We say that ${}^eg$ is projectively compact if near $\p \oM_e$ it has the
form 
\[
{}^eg=\frac{dr^2}{4r^2}+\frac{k}{r},
\]
where $r$ is a defining function for $\p \oM_e$ and $k$ is a smooth 
symmetric 2-tensor on $\oM_e$ which is positive definite when restricted to 
$T\p \oM_e$.  (The papers \cite{capgovercrelles}, \cite{capgovermathann}
consider more general notions of projective compactness; our   
projectively compact metrics are projectively compact of order $2$ in the
terminology introduced there.)  It is easily checked that this class of
metrics is  
independent of the choice of defining function $r$.  Elementary
calculations (see \eqref{chat} below) show that if ${}^e\n$ is the
Levi-Civita connection of such a  
metric and $r$ a defining function, then the connection $\hn$ defined by 
\begin{equation}\label{smoothconnection}
\hn={}^e\n+D,\qquad D_{ij}^k=v_{(i}^{\phantom{k}} \delta_{j)}^k,\qquad v=dr/r 
\end{equation}
extends smoothly up to $\p \oM_e$.
Thus ${}^e\n$ is projectively equivalent to the smooth connection 
$\hn$.  It turns out that projectively compact metrics are the same as 
even asymptotically hyperbolic metrics upon changing the smooth structure
at the boundary.  We digress to formulate the notion of an even structure
on a manifold with boundary, which underlies this equivalence.

Set $\overline{\R^{n+1}_+}=\{(\r,s):\r\geq 0, s\in \R^n\}$.  View  
$\R^n\subset \overline{\R^{n+1}_+}$ as the subset $\r=0$.  
\begin{definition}
Let $\cU\subset \overline{\R^{n+1}_+}$ be open.  Let $f:\cU\rightarrow \R$
be smooth.  $f$ is said to be even (resp. odd) if either:
\begin{enumerate}
\item
$\cU\cap \R^n= \emptyset$, or
\item
$\cU\cap \R^n\neq \emptyset$ and 
the Taylor expansion of $f$ at each point of $\cU\cap \R^n$ has only even
(resp. odd) terms in $\r$. 
\end{enumerate}
\end{definition}

\noindent
It is equivalent to say that $f$ is even (resp. odd) if there is a smooth
function $u$ so that $f(\r,s)=u(\r^2,s)$ (resp. $f(\r,s)=\r\, u(\r^2,s)$).   
A smooth map $\varphi: \cU\rightarrow \overline{\R^{n+1}_+}$ is said to   
be even if it is of the form $\varphi(\r,s) = (\r',s')$, where $\r'$ is odd
and each component of $s'$ is even.     

\begin{definition}
Let $(\oM,\partial \oM)$ be a manifold with boundary, with atlas   
$\{(\cU_\al,\varphi_\al)\}_{\al\in \cA}$.  Let  
$\{(\cU_\al,\varphi_\al)\}_{\al\in\cAt}$ be a subatlas of  
$\{(\cU_\al,\varphi_\al)\}_{\al\in \cA}$ corresponding to a subset
$\cAt\subset \cA$.  We say that  
$\{(\cU_\al,\varphi_\al)\}_{\al\in\cAt}$ defines an even 
structure on $(\oM,\partial \oM)$ subordinate to its smooth structure if 
the transition map 
\[
\varphi_{\al_2}\circ\varphi_{\al_1}^{-1}
:\varphi_{\al_1}(\cU_{\al_1}\cap\cU_{\al_2})
\rightarrow \varphi_{\al_2}(\cU_{\al_1}\cap\cU_{\al_2})
\] 
is even for all $\al_1,\al_2\in \cAt$.  The even structure is defined to be
the maximal atlas containing $\{(\cU_\al,\varphi_\al)\}_{\al\in\cAt}$ for
which all transition maps are even.     
\end{definition}

\begin{remark}
Since $\{(\cU_\al,\varphi_\al)\}_{\al\in\cAt}$ is in particular an atlas
for the smooth structure determined by
$\{(\cU_\al,\varphi_\al)\}_{\al\in\cA}$, the even structure determines the
smooth structure with respect to which it is subordinate.  So there is
really no need to begin with the original smooth structure.  Nevertheless, 
we will usually have the smooth structure to start with and this language
is appropriately suggestive.  There are many different even  
structures subordinate to a given smooth structure.   
\end{remark}

A diffeomorphism for some $\ep>0$ between a collar neighborhood of $\p \oM$
in $\oM$ and $[0,\ep)\times \p \oM$ induces an even structure on 
$(\oM,\p \oM)$.  In fact, an atlas for $\p\oM$ induces an atlas for  
$[0,\ep)\times \p \oM$ whose transition maps are 
the identity in the $\rho$ factor and independent of $\rho$ in the $\p \oM$
factor.    

If $(\oM,\partial \oM)$ is a manifold with boundary with subordinate even  
structure, it is invariantly defined to say that a function $f$ on $\oM$ is 
even:  $f\circ \varphi_\al^{-1}$ is required to be even on 
$\overline{\R^{n+1}_+}$ for
all charts $(\cU_\al,\varphi_\al)$ in the even structure.  Likewise for odd
functions.  Conversely, knowledge of the even and odd functions on
$(\oM,\partial \oM)$ 
 determines the subordinate even structure.       

As an aside, we comment that if $(\oM,\p\oM)$ is a manifold with boundary,
there is a natural one-to-one correspondence between smooth doubles of
$(\oM,\p\oM)$ and subordinate even structures.  Recall that a smooth double
of $(\oM,\p\oM)$ is a choice of smooth manifold structure on the
topological double $2\oM = (\oM\sqcup \oM)/\p\oM$ such that the inclusions 
$\oM\to 2\oM$ are diffeomorphisms onto their range and such that the
natural reflection $2\oM\to 2\oM$ is a diffeomorphism.  The even
(resp. odd) functions on $(\oM,\partial \oM)$ are determined by the double
by the requirement that their reflection-invariant (resp. anti-invariant)
extension to $2\oM$ is smooth.

Denote by $S:\overline{\R^{n+1}_+}\rightarrow \overline{\R^{n+1}_+}$ the
squaring map  
\[
S(\r,s)=(\r^2,s).
\]
Let $(\oM,\partial \oM)$ be a manifold with boundary and let  
$\{(\cU_\al,\varphi_\al)\}_{\al\in\cAt}$ define an even structure on 
$(\oM,\partial \oM)$ subordinate to its smooth structure.  We construct 
another manifold with boundary $(\oM_e,\partial \oM_e)$ as follows.  Set  
$\oM_e=\oM$ as topological spaces.  Define
\[
\psi_\al =S\circ \varphi_\al,\qquad \al\in\cAt.
\]
If $\al_1$, $\al_2\in \cAt$, then 
\reversemarginpar
\begin{equation}\label{eq:even_diffeo}
	(\varphi_{\al_2}\circ\varphi_{\al_1}^{-1})(\r,s)=(\r\, a(\r,s ),s'(\r,s)), 
\end{equation}
where $a$ and the components of $s'$ are even.  
Now 
$\psi_{\al_2}\circ\psi_{\al_1}^{-1}=S\circ
(\varphi_{\al_2}\circ\varphi_{\al_1}^{-1})\circ S^{-1}$.   
Hence 
\[
\begin{split}
(\psi_{\al_2}\circ\psi_{\al_1}^{-1})(r,s)
=&\big(S\circ (\varphi_{\al_2}\circ\varphi_{\al_1}^{-1})\big)(\sqrt{r},s)\\
=&S\big(\sqrt{r}\,a(\sqrt{r},s),s'(\sqrt{r},s)\big)\\
=&\big(r a(\sqrt{r},s)^2,s'(\sqrt{r},s)\big).
\end{split}
\]
Since $a$ and the components of $s'$ are even, it follows that  
$\psi_{\al_2}\circ\psi_{\al_1}^{-1}$ is smooth.  Hence the charts 
$\{(\cU_\al,\psi_\al)\}_{\al\in\cAt}$ define a manifold with boundary
structure on the topological space $\oM$, which we denote  
$(\oM_e,\partial \oM_e)$.  As topological spaces we have $\oM=\oM_e$.  On
the interior, the identity $\cI:M\rightarrow M_e$ is a  
diffeomorphism.  Since $\psi_\al\circ\varphi_\al^{-1}=S$ is smooth, it
follows that $\cI:\oM\rightarrow \oM_e$ is smooth.  
But $\cI^{-1}:\oM_e\rightarrow \oM$ is  
not smooth since in the charts $\psi_\al$, $\varphi_\al$, its first 
component is the function $\sqrt{r}$ on $\overline{\R^{n+1}_+}$.  The
process of passing from $(\oM,\partial \oM)$ with its subordinate even
structure to  
$(\oM_e,\partial \oM_e)$ could be called ``introducing $r=\r^2$ as a new boundary defining function''.        

Next consider the inverse process of ``introducing $\r=\sqrt{r}$ as a new  
boundary defining function''.  Let $(\oN,\partial \oN)$ be any manifold with boundary.  We
construct another manifold with boundary $(\oM,\partial \oM)$ with  
subordinate even structure, such that
$(\oN,\partial \oN)$ equals $(\oM_e,\partial \oM_e)$ as manifolds with
boundary.  
To do so, let 
$\{(U_\al,\psi_\al)\}_{\al\in \cA}$ be an atlas for $(\oN,\partial \oN)$.
Take $\oM=\oN$ as topological spaces.  Use as charts on $\oM$ the maps 
$\varphi_\al=S^{-1}\circ \psi_\al$.  
Now 
\[
(\psi_{\al_2}\circ\psi_{\al_1}^{-1})(r,s)=(r b(r,s),s'(r,s)) 
\]
where $b$ and $s'$ are smooth.  Calculating the compositions as above gives   
\[
(\varphi_{\al_2}\circ\varphi_{\al_1}^{-1})(\r,s)=\left(\r\sqrt{b(\r^2,s)},s'(\r^2,s)\right).   
\]
Since $b(0,s)\neq 0$, this is an even diffeomorphism.  The atlas 
$\{(U_\al,\varphi_\al)\}_{\al\in \cA}$ thus defines the desired 
manifold with boundary $(\oM,\partial \oM)$ with subordinate even structure.
In this case the subatlas $\cAt$ equals $\cA$.

Suppose now that $g$ is an AH metric on the interior $M$ of a compact
manifold with boundary $(\oM,\p\oM)$ with a subordinate even
structure.  In the context of this discussion it is natural to define $g$
to be even relative to the chosen even structure if in
coordinates $(\r,s)$ in the even structure it has the form
\begin{equation}\label{evenAH}
g=\r^{-2}\big(\gb_{00}d\r ^2+2 \gb_{0\a}d\r ds^\a +\gb_{\a\b}ds^\a ds^\b\big)
\end{equation}
with $\gb_{00}$, $\gb_{\a\b}$ even and $\gb_{0\a}$ odd.  The choice of  
a representative $h$ for the conformal infinity induces a diffeomorphism  
between $[0,\ep)\times \p\oM$ and a collar neighborhood of $\p\oM$ with 
respect to which $g$ has the form \eqref{eq:normal_form} with $h_0=h$.   
By analyzing the construction of the normal
form in \cite{MR1112625}, 
it is not hard to see that this diffeomorphism putting $g$ into normal 
form is even relative to the coordinates $(\r,s)$ and the even structure
determined by the product $[0,\ep)\times \p\oM$
(see the proof of \cite[Lemma 2.1]{MR2153454} 
for the special case when \eqref{evenAH} is already in normal form relative
to another representative).  It follows that $g$ is 
even as defined in the introduction and that $g$ uniquely determines the
even structure with respect to which it is even.  
In the other direction, an even AH metric in the sense of the introduction
is clearly even with respect to the even structure determined by any of its
normal forms.  Thus an AH metric $g$ is even in the sense of the introduction
if and only if it is even relative to some even structure subordinate to
the smooth structure on $(\oM,\p\oM)$, and this even structure is uniquely
determined by $g$.

If $g$ is an even AH metric, we can consider the smooth manifold with
boundary $(\oM_e,\p\oM_e)$ obtained from the even structure determined by
$g$ upon introducing $r=\rho^2$ as a new boundary defining function.
Since $\cI^{-1}:M_e\rightarrow M$ is a diffeomorphism,  
${}^eg:=(\cI^{-1})^*g$ is a metric on $M_e$.  We claim that ${}^eg$ is
projectively compact relative to the smooth structure on 
$(\oM_e,\partial \oM_e)$.  
In fact, if $g$ has the form \eqref{eq:normal_form} on  
$[0,\ep)\times \p\oM$ with $h_\rho$ even in $\rho$, then 
\begin{equation}\label{g_e normalform} 
{}^eg=\frac{dr^2}{4r^2} + \frac{k_r}{r}, 
\end{equation}
where $k_r=h_{\sqrt{r}}$ is a one-parameter family of metrics on $\p\oM_e$
which is smooth in $r$.  Thus ${}^eg$ is projectively compact.  
Conversely, a projectively compact metric relative to 
$(\oM_e,\partial \oM_e)$ is an even AH metric when viewed relative to
$(\oM,\partial \oM)$.  

In summary, the class of even asymptotically hyperbolic metrics on the
interior of a manifold with boundary $(\oM,\p\oM)$ with subordinate even 
structure is exactly the same as the class of projectively compact metrics
in the interior of $(\oM_e,\p\oM_e)$.  The distinction is     
just a matter of which smooth structure one chooses to use at infinity.
The smooth structures are related by introducing $r=\rho^2$ as a new  
defining function.

\section{Local Injectivity for Even Metrics}\label{sec:proof_in_even_case} 

Let $(\oM,\p\oM)$ be a manifold with boundary and $g$ an even AH metric on
$M$.  As described in Section~\ref{sec:even_structure_definition}, the
associated metric ${}^eg$ obtained by introducing $r=\r^2$ as a new
defining function is projectively compact.  In particular, for any defining function
$r$ for  $\p\oM_e$, the connection $\hn$ defined by 
\eqref{smoothconnection} is smooth up to $\p\oM_e$.  We will reduce the
analysis of the local X-ray transform of $g$ to that for $\hn$.  

\begin{lemma}\label{lm:cnvx_bdry}
$\p\oM_e$ is strictly convex with respect to $\hn$. 
\end{lemma}
\begin{proof}
Recall that this means that 
if $r$ is a defining function for $\p\oM_e$ with $r>0$ in $M_e$ and if   
$\hg$ is a nonconstant geodesic of $\hn$ such that $r(\hg(0))=0$ and      
$dr\big({\hg}\,'(0)\big)=0$,
then $\p_\t^2(r\circ\hg)|_{\t=0}<0$.  Write $g$ in normal form  
\eqref{eq:normal_form} relative to a conformal representative $h$ on
$\p\oM$, so that ${}^eg$ has the form \eqref{g_e normalform} 
on $M_e$.  Letting $\hG^k_{ij}$ (resp. ${}^e\G_{ij}^k$) denote the
Christoffel symbols of $\hn$ (resp. the Christoffel symbols of
the Levi-Civita connection ${}^e\n$ of ${}^eg$) an easy calculation (see
\eqref{chat} below) shows that ${}^e\G^0_{\al\be}=2k_{\al\be}=2h_{\al\be}$
on $\p\oM_e$.  Since $D^0_{\al\be}=0$, we have at $\t=0$:    
\[
\p_\t^2(r\circ \hg) = -\hG^0_{ij}{\hg}^i{}'\,{\hg}^j{}'\,
= -\hG^0_{\al\be}{\hg}^{\al}{}'\,{\hg}^{\be}{}'\,
=-{}^e\G^0_{\al\be}{\hg}^{\al}{}'\,{\hg}^{\be}{}'\,
=-2h_{\al\be}{\hg}^{\al}{}'\,{\hg}^{\be}{}'\,<0.  
\]
\end{proof}

It will be convenient to embed $\oM_e$ in a smooth compact manifold without 
boundary $\tM$ and to extend $\hn$ to a smooth connection on $\tM$, also 
denoted $\hn$.  If $\hg$ is a geodesic of $\hn$ with $\hg(0)\in \oM_e$, 
set $\t_\pm(\hg):=\pm\sup\{{\t\geq 0}:\hg(t)\in \oM_{e}  
\text{ for } 0\leq \pm t\leq  \t\}$.  
If $U\subset \oM_e$ (usually a small neighborhood of $p\in\p\oM$ or its
closure), we define the set $\hO_U$    
of \textit{$U$-local geodesics of $\hn$} by  
\begin{equation}\label{eq:omega_local_geodesics}
\hO_U:=\big\{\hg:
|\t_\pm(\hg)|<\infty,\;\; |\t_+(\hg)|+ |\t_-(\hg)|> 0,\;\;
     \hg(t)\in U\text{ for }t\in [\t_-(\hg),\t_+(\hg)]\big\}. 
\end{equation}
Here the requirement $|\t_+(\hg)|+ |\t_-(\hg)|> 0$  excludes geodesics
tangent to $\p \oM_e$.  

If $f\in C(\oU)$, set
\begin{equation}\label{eq:defn_of_x_ray}
\hI f(\hg)=\int_{\t_-(\hg)}^{\t_+(\hg)}f(\hg(\t))\,d\t,\qquad \hg\in
\hO_U.   
\end{equation}
The $U$-local X-ray transform of $f$ is the collection of all  
$\hI f(\hg)$, $\hg\in \hO_U$.  

Recall that the parametrization of a geodesic of any connection on 
$TM_e$ is determined up to an affine change   
$\tau\to a\tau + b$, $a\neq 0$.  Such a reparametrization changes
$\hI f(\hg)$ by multiplication by $|a|^{-1}$.  In particular, whether or not    
$\hI f(\hg)=0$ is independent of the parametrization.  It suffices to
restrict attention to geodesics whose parametrization satisfies a
normalization condition.  
For instance, in the next section we fix a background metric $g^0$ and
require that $|{\hg}'(0)|_{g^0}=1$.   

Next we relate $I$ and $\hI$.  This involves relating objects on   
$M$ with objects on $M_e$.  Since $\cI:M\to M_e$ is the identity map, this  
amounts to viewing the same object in a different smooth structure, i.e. in
different coordinates near the boundary.  We  
suppress writing explicitly the compositions with the charts $\psi_\al$,
$\varphi_\al$.  So the 
expression of the identity in these coordinates is $\cI(\r,s)=(\r^2,s)$.       
Likewise, $g$ and ${}^eg$ are related in coordinates by setting $r=\r^2$,
as in \eqref{g_e normalform}.  If $f$ is a function defined on $M$, we can
regard $f$ as a function $f_e$ on $M_e$, related in coordinates by 
$f(\r,s)=f_e(\r^2,s)$.   If $U\subset M$, set $U_e = \cI(U)$.   

If $\g(t)$ is a $U$-local geodesic for $g$,
it is also a geodesic for ${}^eg$.  Since ${}^e\n$ is projectively
equivalent to $\hn$, \eqref{parameterchange} and \eqref{smoothconnection} 
imply that $\hg(\t):=\g(t(\t))$ is a geodesic for $\hn$, where 
$dt/d\t = c\big(r\circ\g(t(\t))\big)^{-1}$.  Different choices of $c$ determine
different parametrizations; imposition of a normalization condition on the
parametrization as mentioned above provides one way to specify $c$ 
for each geodesic.  The relation between $I$ and $\hI$ follows easily:  
\begin{equation}\label{Irelation}
If(\g)=\int_{-\infty}^\infty f(\g(t))\,dt 
=|c|\int_{\t_-(\hg)}^{\t_+(\hg)} (r^{-1}f_e)(\g(t(\t)))\,d\t
=|c|\hI(r^{-1}f_e)(\hg).  
\end{equation}

Section 3.4 of \cite{Uhlmann2016} shows that if $U_e$
is a sufficiently small open neighborhood of $p\in \p\oM_e$, then the
$U_e$-local X-ray transform for a smooth metric extends to a bounded
operator on $L^2(U_e)$ with target space $L^2$   
of a parametrization of the space of $U_e$-local geodesics with respect to
a suitable measure.  The same argument holds in our setting for a smooth   
connection such as $\hn$.  We will not make explicit the target $L^2$ space  
since we are only concerned here with injectivity.  

Equation \eqref{Irelation} shows that it is important to understand when
$r^{-1}f_e\in L^2(U_e)$.  Making the change of variable $r=\r^2$ in the
integral gives 
\[
\int (r^{-1}f_e)^2\,drds
=2\int (\r^{-2}f)^2\r\,d\r ds
=2\int (\r^{-3/2}f)^2\,d\r ds. 
\]
Thus $r^{-1}f_e\in L^2(U_e,drds)$ if and only if  
$f\in \r^{3/2}L^2(U,d\r ds)$.  In particular, 
$If(\g)=|c|\hI(r^{-1}f_e)(\hg)$ provides a definition of 
$If$ for $f\in \r^{3/2}L^2(U,dv_{\og})$ consistent with its usual
definition.  

The main result of \cite{Uhlmann2016} is local injectivity of the geodesic 
X-ray transform for a smooth metric on a manifold with strictly convex 
boundary.  However, the proof applies just as well for the X-ray 
transform for a smooth connection such as $\hn$.  In particular, the
construction in the main text of the cutoff function $\chi$ for which the
boundary principal symbol is elliptic is also valid 
for a connection since the right-hand side of the geodesic  
equation ${\g}^k\,''=-\G_{ij}^k{\g}^i\,'{\g}^j\,'$ is a quadratic
polynomial in ${\g}'$.  We do not need the extension of Zhou discussed
in the appendix of \cite{Uhlmann2016}, although that more general result
applies as well.  The main result of \cite{Uhlmann2016} transferred to our
setting is as follows.   
\begin{theoremofothers}[\cite{Uhlmann2016}]\label{UV1}
Assume that $\dim \oM_e\geq 3$ and let $p\in \p\oM_e$.  Every neighborhood
$O_e$ of $p$ in $\oM_e$ contains a neighborhood $U_e$ of $p$ so that the
$U_e$-local X-ray transform of $\hn$ is injective on $L^2(U_e)$.    
\end{theoremofothers}

\begin{proof}[Proof of Theorem~\ref{thm:main} for $g$ even]
The relation \eqref{Irelation} shows that 
$f\in \r^{3/2}L^2(U,dv_{\og})$ is in the kernel of the $U$-local transform
for $g$ if and only if 
$r^{-1}f_e\in L^2(U_e)$ is in the kernel of the $U_e$-local transform for 
$\hn$.  Thus for $g$ even, Theorem~\ref{thm:main} follows immediately from
Theorem~\ref{UV1}. 
\end{proof}

\section{\texorpdfstring{Connections Associated to AH Metrics Even mod
  $O(\r^N)$}{Connections Associated to AH Metrics Even mod
  O(rho N)}}\label{nonsmooth}

If the AH metric $g$ in \eqref{eq:normal_form} is not even, then the
even structure on $(\oM,\p\oM)$ determined by a normal form for $g$ depends
on the choice of normal form.  We fix one such
normal form and thus the even structure  it determines.  We then
construct $(\oM_e,\p\oM_e)$ as above by introducing $r=\rho^2$ as a new
boundary defining function.  The metric ${}^eg$ would be projectively
compact except that the  
corresponding one-parameter family $k_r=h_{\sqrt{r}}$ in 
\eqref{g_e normalform} is no longer smooth: it has an expansion in powers
of $\sqrt{r}$.  The connection $\hn$ defined by \eqref{smoothconnection} 
involves first derivatives of $k_r$.  
As already discussed in the Introduction, assuming that $g$ is even mod
$O(\r^5)$ suffices to guarantee that $\hn$ is Lipschitz continuous, and, in
fact,  
that it extends to be $C^1$ up to $\p\oM_e$, though not necessarily $C^2$.    
Near $\p \oM_e$,  $\hn$ can be viewed as a perturbation of a smooth connection $\on$.

Straightforward calculation from \eqref{g_e normalform} shows that the
Christoffel symbols of the connection $\hn$ defined by
\eqref{smoothconnection} are given in terms of coordinates near a point
$p\in \p \oM_e$ by 
\begin{gather}\label{chat}
\begin{gathered}
\hG_{ij}^0 = 
\begin{pmatrix}
0&0\\
0&2(k_{\al\be}-r\p_rk_{\al\be})
\end{pmatrix},\qquad
\hG_{ij}^\ga = 
\begin{pmatrix}
0&\tfrac12 k^{\ga\de}\p_rk_{\de\be}\\
\tfrac12 k^{\ga\de}\p_rk_{\al\de}&\G_{\al\be}^\ga 
\end{pmatrix},
\end{gathered}
\end{gather}
where $\G_{\al\be}^\ga$ denotes the Christoffel symbols of $k_r$ with
$r$ fixed.  If $g$ is even mod $O(r^N)$ with $N$ odd, then 
$k=k^{(1)}+r^{N/2}k^{(2)}$ with $k^{(1)}$, $k^{(2)}$ smooth.  It    
follows that all $\hG_{ij}^k$ have the form
\[
\hG_{ij}^k=\oG_{ij}^k + r^{N/2-1}B_{ij}^k
\]
with $\oG_{ij}^k$, $B_{ij}^k$ smooth up to $\p\oM_e$.  
The expressions $\oG_{ij}^k$, $B_{ij}^k$ can be interpreted as the
Christoffel symbols of a smooth connection $\on$ on $\oM_e$ and the
coordinate expression of a $(1,2)$ tensor field $B$ respectively. 
$\on$ and $B$ are not uniquely determined by the connection $\hn$; henceforth we fix one choice for them.
Recall that we have chosen a closed manifold $\tM$ containing $\oM_e$.
Choose some smooth extension of $\on$ to a neighborhood of $\oM_e$, also
denoted 
$\on$.  Then extend $\hG$ by 
\begin{equation}\label{defhG}
\hG_{ij}^k=\oG_{ij}^k + r^{N/2-1}H(r)B_{ij}^k
\end{equation}
where $H(r)$ is the Heaviside function.  The extended connection $\hn$ is
then $C^{(N-3)/2}$ and the two connections $\hn$, $\on$ agree outside of $M_e$.  

An important consequence of the special structure of the connection $\hn$ is
 that its exponential map is more regular than one would 
expect.  We consider the exponential map in the form  
$\hexp:T\tM\to \tM\times \tM$, 
defined by $\hexp(z,v)=(z,\hphi(1,z,v))$, where $t\to \hphi(t,z,v)$ is the 
geodesic with $\hphi(0,z,v)=z$, $\hphi\,'(0,z,v)=v$.  
Since $\hn$ is $C^{(N-3)/2}$ and $N\geq 5$, usual ODE theory implies that
$\hexp$ is a $C^{(N-3)/2}$ diffeomorphism from a neighborhood of the zero
section onto its image.  In fact, it has one more degree of
differentiability.  We formulate the result in terms of the inverse
exponential map since that is how we will use it.  
\begin{lemma}\label{lm:regularity}
Let $\hn$ be the $C^{(N-3)/2}$ connection  defined by
\eqref{defhG}, where $N\geq 5$ is an odd integer. 
Then $\hexp^{-1}$ is $C^{(N-1)/2}$ in a neighborhood in $\tM\times \tM$ of the diagonal in $\p\oM_e\times \p\oM_e$.  
\end{lemma}

\begin{proof} 

It suffices to show that $T\tM\ni (z,v)\to \hphi(1,z,v)\in \tM$ is 
$C^{(N-1)/2}$ near $(z,0)$ for $z\in \p\oM_e$.  
Work in coordinates $(r,s)$ 
for $z$ with respect to which ${}^eg$ is in normal form 
\eqref{g_e normalform}.  Set $z=(z^0,z^\a)=(r,s^\a)$.  For $v$ use 
induced coordinates $v=(v^0,v^\a)$ with
$v=v^0\p_r+v^\a\p_{s^\a}=v^i\p_{z^i}$ and set $w=(z,v)$.  
Write the flow as 
$\hphi(t,w)=(\td{z}(t,w),\td{v}(t,w))$.
The geodesic equations are:  
\begin{equation}\label{geodeq}
(\td{z}\,^k)'=\td{v}\,^k,\qquad 
(\td{v}\,^k)'=-\hG_{ij}^k(\td{z})\td{v}\,^i\td{v}\,^j.
\end{equation}
Observe from \eqref{chat} that all $\hG_{ij}^k$ are $C^{(N-1)/2}$ except
for $\hG_{0\al}^\ga=\hG_{\al 0}^\ga$.  So the right-hand sides of all
equations in \eqref{geodeq} are $C^{(N-1)/2}$ except for the equation for 
$(\td{v}\,^\ga)'$.  By \eqref{chat}, \eqref{defhG}, this equation has the
form 
\begin{equation}\label{badone}
(\td{v}\,^\ga)'= A_{ij}^\ga(\td{z})\td{v}\,^i\td{v}\,^j
-2\td{r}^{N/2-1}H(\td{r})B_{0\be}^\ga(\td{z})\td{v}\,^0\td{v}\,^\be 
\end{equation}
with $A_{ij}^\ga$ of regularity $C^{(N-1)/2}$ and $B_{0\b}^\ga$ smooth.
Using $\td{r}\,'=\td{v}\,^0$, write
\begin{align}
-2\td{r}^{N/2-1}H(\td{r})B_{0\be}^\ga(\td{z})\td{v}\,^0\td{v}\,^\be 
=&-\frac{4}{N}\Big(\td{r}^{N/2}H(\td{r})\Big)'B_{0\b}^\ga(\td{z})\,\td{v}\,^\b\\   
=&-\frac{4}{N}\Big(\td{r}^{N/2}H(\td{r})B_{0\b}^\ga(\td{z})\,\td{v}\,^\b\Big)'\\  
&\quad +
\frac{4}{N}\td{r}^{N/2}H(\td{r})\Big(B_{0\b}^\ga{}_{,k}(\td{z})\,
\td{v}\,^k\td{v}\,^\b+B_{0\b}^\ga(\td{z})\,(\td{v}\,^\b)'\Big)\\  
=&-\frac{4}{N}\Big(\td{r}^{N/2}H(\td{r})B_{0\b}^\ga(\td{z})\,\td{v}\,^\b\Big)'
+\frac{4}{N}\td{r}^{N/2}H(\td{r})C_{ij}^\g(\td{z})\td{v}\,^i\td{v}\,^j,
\end{align}
where for the last equality we have used \eqref{geodeq} for
$(\td{v}\,^\b)'$, so that
\[
C_{ij}^\g(\td{z})\td{v}\,^i\td{v}\,^j
=B_{0\b}^\ga{}_{,k}(\td{z})\,\td{v}\,^k\td{v}\,^\b
-B_{0\b}^\g(\td{z})\hG^{\b}_{ij}(\td{z})\td{v}^i\td{v}^j.
\]
Note that 
$\td{r}^{N/2}H(\td{r})C_{ij}^\g(\td{z})\td{v}\,^i\td{v}\,^j$ is
$C^{(N-1)/2}$.

Therefore \eqref{badone} can be rewritten in the form
\begin{equation}\label{goodone}
\Big(v^\g+\frac{4}{N}\td{r}^{N/2}H(\td{r})B_{0\b}^\ga(\td{z})\,\td{v}\,^\b\Big)' 
=\Big(A_{ij}^\ga(\td{z})
+\frac{4}{N}\td{r}^{N/2}H(\td{r})C_{ij}^\g(\td{z})\Big)\td{v}\,^i\td{v}\,^j.
\end{equation}
Now the linear transformation $\td{v}\mapsto \td{b}=L(\td{z})\td{v},$ where
$\td{b}\,^\g=\td{v}\,^\g+\frac{4}{N}\td{r}^{N/2}H(\td{r})B_{0\b}^\g(\td{z})\,\td{v}\,^\b$,  
is of class $C^{(N-1)/2}$ in $(\td{z},\td{v})$ and 
is invertible for $\td{r}$ small.  
Replacing \eqref{badone} by \eqref{goodone} in \eqref{geodeq} and setting 
$\td{v}=L^{-1}(\td{z})\td{b}$ throughout, we obtain 
a system of ODE of the form 
\begin{equation}\label{eq:reqwritten_ode} 
\Big(\td{z},\td{v}\,^0,\td{b}\Big)'
=F\Big(\td{z},\td{v}\,^0,\td{b}\Big), 
\end{equation}
where $F$ is $C^{(N-1)/2}$.
It follows that the map $(t,z,v)\mapsto \widehat{\phi}(t,z,v)$ is of class 
$C^{(N-1)/2}$ upon setting $\td{b}\,^\g=L^\g_\b(\td{z})\td{v}\,^\b$.   
\end{proof}

Lemma~\ref{lm:cnvx_bdry} (the strict convexity of $\p\oM_e$) holds for  
both $\hn$ and $\on$ if $g$ is even mod $O(\r^N)$ with $N\geq 5$
odd, with the same proof as before.  We define the sets $\hO_U$, $\oO_U$ of
$U$-local geodesics for $\hn$ and $\on$ the same way as before.  It will be
important to have a common parametrization for the sets of geodesics of $\hn$
and $\on$.  For this purpose, we will fix a smooth background metric $g^0$
on $\tM$ (this will be done  in Section \ref{estimates}).   
There is no canonical way of choosing $g^0$ and the choice made does not
affect  the conclusions, but a convenient choice will simplify some
computations. 
Once a metric $g^0$ has been fixed, we let $S^0\tM$ denote  
its unit sphere bundle.  For $v\in S^0\tM$, denote by $\hg_v$,
(resp. $\oga_v$) the geodesic for $\hn$ (resp. $\on$) with initial vector
$v$. 
 We define the $U$-local X-ray transforms for $\hn$ and $\on$ just as
in \eqref{eq:defn_of_x_ray}, except now we view them as functions on the
subsets of $S^0\tM$ corresponding to $\hO_U$, $\oO_U$:  
\[
\hI f(v)=\int_{\t_-(\hg_v)}^{\t_+(\hg_v)}f(\hg_v(\t))\,d\t   
\]
and similarly for $\oI f(v)$.  Sometimes we will use the notation $If(v)$
generically for $\hI f(v)$ or $\oI f(v)$, or, for that matter, for the
$U$-local X-ray transform for any $C^1$ connection on a manifold with
strictly convex boundary.  No confusion will arise with the notation
$If(\g)$ from Section~\ref{sec:proof_in_even_case} for the X-ray transform
for the AH metric $g$, since we will not be dealing with $g$ again except 
implicitly in the isolated instance where we deduce Theorem~\ref{thm:main}.

\section{Stability and Perturbation Estimates}\label{estimates}

We continue to work with the connections $\hn$ and $\on$ obtained from an
AH metric even mod $O(\r^N)$ with $N\geq 5$.
From now on it will always be assumed that the dimension of $\oM$ (and thus also of $\oM_e$) is at least 3.
  Since $\on$ is smooth and
$\p\oM_e$ is strictly convex with respect to it, 
Theorem~\ref{UV1} (local injectivity) holds also for $\on$.  
As mentioned in the Introduction, in order to
deduce local injectivity for $\hn$ we will use the stability estimate
derived  in \cite{Uhlmann2016} for the conjugated microlocalized normal
operator $\oA_{\chi,\eta,\s}$, formulated in terms of 
scattering Sobolev spaces.  In this section we review those spaces,
the construction of the microlocalized normal operator, and the stability
estimate proved in \cite{Uhlmann2016}. 
Then we formulate our main perturbation estimate 
(Proposition~\ref{prop:mapping_diff}) and show how Theorem~\ref{thm:main}
follows from it. 
Proposition~\ref{prop:mapping_diff} will be proved in Section~\ref{kernelanalysis}.  
In this section we work almost entirely on $\oM_e$  and its extension $\tM$
(with the exception of the very last proof), so we will not be using the
subscript $e$ for its various subsets to avoid cluttering the notation.

We now define 
polynomially weighted scattering Sobolev spaces on a 
compact manifold with boundary $(\oX^{n+1},\p\oX)$.
Let $x$ be a boundary defining function for $\oX$.
The space of \textit{scattering vector fields}, denoted by
$\calV_{\Sc}(\oX)$, consists of the smooth vector fields on $\oX$
which are a product of $x$ and a smooth vector field tangent to $\p \oX$. 
Thus  if $(x,y^1,\dots, y^n)$ are coordinates near $p\in\p \oX$,  elements
of $\calV_{\Sc}(\oX)$ can be written near $p$ as linear combinations over
$C^\infty(\oX) $ of the vector fields $x^2\p_x,$ $x\p_{y^\a}$,
$\a=1,\dots, n$.  If $k\in \mathbb{N}_0$ and $\b\in \R$  let
\begin{equation}
\begin{aligned}\label{eq:sobolev}
  &H_{\Sc}^{k,\b}(\oX) =\{u\in
  x^{\b}L^2(\oX):x^{-\b}V_1\dots 
  V_mu\in L^2(\oX)\text{ for } V_j\in \calV_{\Sc}(\oX)\text{ and
  }0\leq m\leq k\}; 
\end{aligned}
\end{equation}
here $L^2$ is
 defined using a smooth measure on $\oX$
\footnote{Our notation slightly differs from that of \cite{Uhlmann2016} in
  that we use a smooth measure rather the scattering measure
  $x^{-(\dim\o{X}+1)}dxdy$ to define our base $L^2$ space. The spaces here
  and in \cite{Uhlmann2016} are the same up to shifting the weight by
  $(\dim\o{X}+1)/2$.}. 
Note that $H^{0,\b}_{\Sc}(\o{X})=x^{\b}L^2(\oX)$.
For $s\geq0$, $H^{s,\b}_{\Sc}(\oX)$ can be defined by interpolation and for
$s<0$ by duality, though we will not need this. 
The norms $ \|\cdot \|_{H_{\Sc}^{k,\b}(\oX)}$ can be defined by  fixing
scattering vector fields in coordinate patches on $\oX$  that locally span
$\calV_{\Sc}(\oX)$ over $C^\infty(\oX)$; 
any different choice of vector fields would result in an equivalent norm.
If $U$ is a neighborhood of $p\in \p \oX$ (or the closure of one) 
then 
$H_{\Sc}^{k,\b}(U)$ consists of functions of the form $u\big|_{U}$, where 
$u\in H_{\Sc}^{k,\b}\left(\oX\right)$.

\smallskip

We next review the arguments and results we
will need from \cite{Uhlmann2016}, starting with the construction of the
artificial boundary mentioned in the introduction. 

\begin{lemma}[\cite{Uhlmann2016}, Sec. 3.1]\label{lm:artificial_boundary}
Let $p\in \p \oM_{e}$ and $\n$ be a $C^1$ connection with respect to which $\p
\oM_{e}$ is strictly convex. 
There exists a smooth function $\hat{x}$ in a neighborhood $\calU$ of $p$ in $\tM $ with the properties:
\begin{enumerate}
    \item  $\hat{x}(p)=0$
    \item $d \hat{x}(p)=-dr(p)$\label{it:second_prop} (recall that $r$ is a boundary defining function for $\oM_e$)
    \item \label{it:third_prop}Setting $x_\eta:=\hat{x}+\eta$, for any neighborhood $\td{O}$ of $p$ in $\tM$
      there exists an $\eta_0$ such that $U_\eta:=\{r\geq 0\}\cap \{\x
      \geq0\}\subset \td{O}$ for $\eta\leq \eta_0$ 
    \item For $\eta$ near 0 (positive or negative) the set
      $X_\eta:=\{\hat{x}>-\eta\}=\{\x>0\}\subset \tM$ has strictly concave
      boundary with respect to $\n$ locally near $p$.
      \footnote{Recall that this means that for any $\n$-geodesic $\g(t)$
        with $x_\eta(\g(0))=0$ and $dx_\eta(\g'(0))=0$ one has
        $\frac{d^2}{dt^2}\big|_{t=0}x_\eta\circ \g(t)>0$. } 
 \end{enumerate} 
 The level sets of $\hat{x}$ can be seen in Figure \ref{figurename_artificial_boundary}.
\end{lemma}

Write $Y_p=\{\hat{x}=0\}$; by shrinking $\calU$ we can assume that $Y_p\cap \calU$ is a smooth hypersurface of $\tM$. 
We can then identify a neighborhood of $p$ in $\tM$ with
$(-\e,\e)_{\hat{x}}\times Y_p$  for some small $\e>0$ via a diffeomorphism
$\phi_0$ (which can be constructed e.g. using the flow of a vector field
transversal to $Y_p$). 
Fixing coordinates $y^1,\dots , y^n$ for $Y_p$ centered at $p$,  choose the
metric $g^0$ so that in  a neighborhood of $p$ it is Euclidean in terms of
coordinates $ (\hat{x},y^1,\dots,y^n)$. 

For  $\calU'$ a small neighborhood of $p$ contained in $\calU$ and $\eta\in
\R$ small, denote by $\psi_\eta:\calU'\to \tM$ the map which in terms of
the above identification maps $(x,y)\mapsto (x+\eta,y)$. 
For a fixed small $0<\d_0\ll\e$ and $0\leq \eta<\d_0$, we can  identify a
neighborhood of  $\p \oX_\eta$ in $\tM $ with $(-\d_0,\d_0)_{x_\eta}\times
Y_p $ via the diffeomorphism  $\phi_\eta=\phi_0\circ \psi_{\eta}$.
Note that $g^0$ is also Euclidean in terms of coordinates $(x_\eta,y^1,\dots,y^n)$.
Moreover, $\oX_\eta$ is  given locally near its boundary by
$[0,\d_0)_{x_\eta}\times Y_p$ in terms of this identification 
and $\psi_{-\eta}$ maps diffeomorphically a neighborhood of $\p \oX_0$ in
$\oX_0$ onto one of $\p \oX_\eta$ in $\oX_\eta$, with inverse
$\psi_{\eta}$. 
Vectors in $S^0_z\tM$, $z\in \o{X}_\eta$, can be written as
$v=\l\p_{\x}+\omega$, where $\omega\in T Y_p$ (of course not necessarily of
unit length, so our setup slightly differs from the one in
\cite{Uhlmann2016}, see Remark \ref{rmk:uv_explanation} below). 
Henceforth, the notation $|v|$  for a vector $v$ will refer to norm with
respect to $g^0$ (which is Euclidean in our coordinates in the region of
interest).

In order to show local injectivity of the X-ray transform, one needs a
description of geodesics staying within a given neighborhood:

\begin{lemma}[\cite{Uhlmann2016}, Section 3.2]
\label{lm:geod_cond}
Let $\n$ be a $C^1$ connection with respect to which $\oM_{e}$ has strictly
convex boundary.
There exist constants  $\td{C}>0$, $0<\d_1<\d_2$, $c_0>0$ and $\eta_0>0$, and neighborhood $Z_p$ of $p$ in $Y_p$,
such that if $0\leq \eta< \eta_0$ and if $\g(t)$ is a $\n$-geodesic with
initial position $z=(x,y)\in [0,c_0]_\x\times \o{Z}_p\subset \oX_\eta$ and
velocity $v=(\l,\omega)\in S^0_z\tM$ satisfying 
\begin{equation}\label{eq:condition_on_initial_velocity}
  \frac{|\l|}{|\omega|}\leq  \td{C}\sqrt{x} %
\end{equation}
then  
one has $\x\circ\g(t)\geq 0$ for $|t|\leq
\d_2$ and $\x\circ\g(t)\geq c_0$ for $|t|\geq \d_1$.
See Fig. \ref{fig:level_sets}.
\end{lemma}

By taking $\eta_0\ll c_0$ in Lemma \ref{lm:geod_cond} and by Lemma
\ref{lm:artificial_boundary} one can always assume that a neighborhood of
$U_\eta$ in $\oX_\eta$ is contained in  $  [0,c_0]_\x\times\o{ Z}_p$, and
we will henceforth assume that this is the case. 
Now let $\d_1$ and $\n$ be as in Lemma \ref{lm:geod_cond} and 
  let $\exp:T\tM\to \tM$ be the exponential 
map of $\n$.  
If $v\in S_z^0\tM$ satisfies
the assumptions of the lemma
and $f$ is continuous and supported in $  [0,c_0)_\x\times{ Z}_p$, we have  
$If(v)=\int_{-\d_1}^{\d_1}f(\exp(tv))dt$, so for all such $v$ and $f$ one
can define the X-ray transform by integrating only over a fixed finite
interval. 
The authors of \cite{Uhlmann2016} consider $If$ only on vectors $v=(\l,\omega)\in S_z^0\tM$ 
satisfying a stronger condition, namely that for some positive constant
$C_2$ one has $\frac{|\l|}{|\omega|}\leq C_2x$ with $z=(x,y)\in [0,c_0)_\x\times {Z}_p$ for $\eta$
sufficiently small, and construct a microlocalized normal operator for $I$.
Specifically, with  $f$ as before and $\chi\in C_c^\infty(\R)$ with $\chi\geq 0$ and
$\chi(0)=1$, let  
\begin{equation}\label{eq:Istar_chi_I}
    {A}_{\chi,\eta}
    f(z):=\int_{S^0_z\tM}\chi\!\left(\frac{\l}{|\omega|x}\right)I
    f(v)d{\mu}_{g^0},\quad z=(x,y)\in [0,c_0)_\x\times {Z}_p,   
\end{equation}
where $d{\mu}_{g^0}$ is the measure induced on $S^0_z\tM$ by
$g^0|_{T_z\tM}$. 
Note that for any $C_2$, $c_0$ can be chosen sufficiently small that
\eqref{eq:condition_on_initial_velocity} is automatically satisfied in $[0,c_0)_\x\times {Z}_p$. 
The constant $C_2$ is fixed when $\chi\in C_c^\infty(\R)$ is chosen (see
Proposition \ref{prop:pdo} below), and then $c_0$, $\eta_0$ can be chosen so that
the integrand in \eqref{eq:Istar_chi_I} is only supported on vectors corresponding to geodesics staying in $\oX_\eta$. 
Finally for $\s>0$ define the conjugated microlocalized normal operator: 
\begin{equation}\label{eq:conjugated_operator}
    {A}_{\chi,\eta,\s}:=x_\eta^{-2}e^{-\s/\x}{A}_{\chi,\eta}e^{\s/\x}.  
\end{equation}
We denote this operator in case $\n=\on$ (resp. $\hn$) by  
$\oA_{\chi,\eta,\s}$ (resp. $\hA_{\chi,\eta,\s}$).  
In the case of the smooth connection $\on$ on $\oM_e$, for which $\p \oM_e$
is strictly convex, and in dimension $\geq 3$, it was proved in
\cite[Proposition 3.3]{Uhlmann2016} that $\oA_{\chi,\eta,\s}$  are
scattering pseudodifferential operators (in the notation there,
$\oA_{\chi,\eta,\s}\in \Psi_{sc}^{-1,0}(\oX_\eta)$). 
This implies that they also act on scattering Sobolev spaces.
The following Proposition contains the stability estimate we will need in terms of such  spaces.

\begin{figure}[ht]
  \includegraphics[scale=.4]{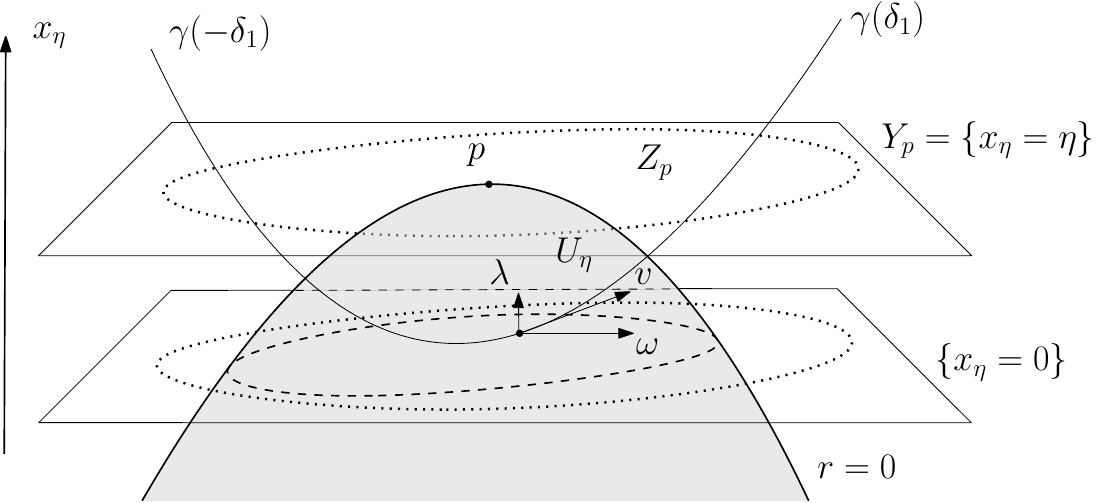}
  \caption{The level sets of $x_\eta$}
  \label{fig:level_sets}
\end{figure}

\begin{proposition}[\cite{Uhlmann2016}, Sections 2.5 and 3.7]\label{prop:pdo}
Suppose as before that $\dim(\oM_e)\geq 3 $ and let $\s>0$. There exists 
$\chi_0\in C^\infty_c(\R)$, $\chi_0\geq 0$, $\chi_0(0)=1$, such that for
any sufficiently small neighborhood $O$ of $p\in \p \oM_e$ in $\oX_0$
there exist $\eta_0>0$ and $C_0>0$ with the property that for $0\leq
\eta\leq \eta_0$ one has   $U_\eta\subset O_\eta:=\psi_{-\eta}(O)\subset
\oX_\eta$, and the estimate 
\begin{equation}\label{eq:stability_o_nabla}
  \|u\|_{x^\b L^2(U_\eta)}\leq
  C_0\|{\oA}_{\chi_0,\eta,\s}u\|_{H^{1,\b}_{\Sc}(O_\eta)}, \quad \b\in \R,
\end{equation}
where $u\in x^{\b}L^2(U_\eta)$ is extended by 0 outside $U_\eta$. 
Here the Sobolev spaces on subsets of $\oX_\eta$ are defined by pulling
back by $\psi_\eta$ the corresponding spaces on subsets of $\oX_0$. 
\end{proposition}

\begin{remark}\label{rmk:uv_explanation}
The estimate stated in  \cite[Section 3.7]{Uhlmann2016} amounts to
\begin{equation}
\|u\|_{H^{s,\b}_{\Sc}(\oX_\eta)}\leq
  C_0\|{\oA}_{\chi_0,\eta,\s}u\|_{H^{s+1,\b}_{\Sc}(\oX_\eta)},\quad s\geq 0,\; \supp u\subset U_\eta,	\label{eq:original_statement}
\end{equation}	
upon taking into account that the analog of $A_{\chi,\eta}$ constructed
there has a factor of $x^{-1}$ incorporated and the polynomial factor 
appearing in the definition of the operator analogous to $A_{\chi,\eta,\s}$
is $x_\eta^{-1}$, whereas we used a factor of $x_\eta^{-2}$ in
$A_{\chi,\eta,\s}$  directly.  For $s=0$ the space on the left hand side of
\eqref{eq:original_statement} is exactly $x^\b L^2(U_\eta)$. 
On the other hand, the upper bound in \eqref{eq:original_statement} can be
replaced by the one in \eqref{eq:stability_o_nabla} provided $\supp
u\subset U_\eta$, since the Schwartz kernel of the operators
$\oA_{\chi_0,\eta,\s}$ has been localized in both factors near $U_\eta$,
see for instance \cite[Remark 3.2]{Uhlmann2016}. 

The way we construct the operators $A_{\chi,\eta}$ also differs from the
setup of \cite{Uhlmann2016} in that we parametrize geodesics by their
initial velocities normalized so that they have unit length with respect to
the (Euclidean near $p$) metric $g^0$, and average the transform over them
using the measure induced by $g^0$ on the fibers of $S^0\tM$ (i.e. the
standard measure on the unit sphere $\S^n$).  
In \cite{Uhlmann2016} the geodesics are parametrized by writing their
initial velocities as $(\l,\omega)\in \R\times \S^{n-1}$ using coordinates,
and the measure used for averaging is $d\l d\omega$, where $d\omega$ is the
standard measure on $\S^{n-1}$. 
However this difference doesn't affect the analysis, as already remarked
there (see Remark 3.1 and the proof of Proposition 3.3). 
\end{remark}

\begin{remark}
As remarked in \cite[Lemma 3.6]{Uhlmann2016}, Proposition~\ref{prop:pdo}
holds for any $\chi_0$ sufficiently close to a specific Gaussian in the
topology of Schwartz space. 
In particular, $\chi_0$ can be taken to be even, and from now on we assume 
that this is the case, since this simplifies the notation.  
\end{remark}

Let $\chi_0$ be as in Proposition \ref{prop:pdo}, chosen to be even.  Let
$\s>0$ be fixed.  Define 
\[
  E_{\eta,\s}:=\oA_{\chi_0,\eta,\s}-\hA_{\chi_0,\eta,\s}
\]
Note that by construction the operator $\oA_{\chi_0,\eta,\s}$
(resp. $\hA_{\chi_0,\eta,\s}$) depends on the behavior of the connection
$\on$ (resp. $\hn$) only in the set $\x\geq 0$, provided $\eta_0$, $c_0$
above are sufficiently small.  
Therefore $E_{0,\s}=0$, since the two connections agree outside of $M_e$.

In Section \ref{ssec:analysis_on_blowups} we will prove 
the following key proposition:

\begin{proposition}\label{prop:mapping_diff}
Let $\s>0$. 
Provided $O$ is a sufficiently small neighborhood of $p\in \pM_e$ in
$\oX_0$, for each $\d>0$  there exits $\eta_0>0$ with the property that if
$0\leq\eta< \eta_0$ one has   $U_\eta \subset O_\eta=\psi_{-\eta}(O)$ and 
\begin{equation}\label{eq:injectivity_estimate}
\|E_{\eta,\s}u\|_{H_{\Sc}^{1,0}(O_\eta)}
\leq \d  \|u\|_{L^2(U_\eta)}
\end{equation}
for all $u\in
 L^2(U_\eta)$ extended by 0 outside of $U_\eta$.

\end{proposition}

\begin{remark}
	In Proposition \ref{prop:mapping_diff} one does not need to assume
        that $\dim(\oM_e)\geq 3$, however if $\dim(\oM_e)=2$ Proposition
        \ref{prop:pdo} does not hold and the proof of Corollary
        \ref{cor:inj_pc_conn} below breaks. 
\end{remark}

\noindent
An immediate consequence of Proposition~\ref{prop:mapping_diff} is the
following:  
\begin{corollary}\label{cor:inj_pc_conn}
With notations as before and assuming that $\dim (\oM_e)\geq 3$,
 there exists $\eta_0>0$ such that for $0<\eta< \eta_0$ the transform 
$f\mapsto \hI f\big|_{\widehat{\Omega}_{U_\eta}}$ is injective on $L^2(U_\eta)$.
\end{corollary}
\begin{proof}
Fix $\s>0$ and let $\chi_0$ be as in Proposition \ref{prop:pdo}, even. Then
take $O$ sufficiently small, as in Propositions \ref{prop:pdo} and
\ref{prop:mapping_diff}, and let $C_0$ and $\eta_0$ be according to the
former, corresponding to $O$. 
By Proposition \ref{prop:mapping_diff}, upon shrinking $\eta_0$ if necessary, for $0\leq\eta<\eta_0$ we have
$$\|E_{\eta,\s}u\|_{H_{\Sc}^{1,0}(O_\eta)}\leq 1/(2C_0)\|u\|_{L^2(U_\eta)}$$
for $u\in L^2(U_\eta)$ extended by $0$ elsewhere.
Since $\oA_{\chi_0,\eta,\s}=\hA_{\chi_0,\eta,\s}+E_{\eta,\s}$, if
$u\in L^2(U_\eta)$ one has,
for $0\leq \eta< \eta_0$ 
\begin{align}
    &\|u\|_{L^2(U_\eta)} \leq  C_0\|\oA_{\chi_0,\eta,\s}u\|_{H^{1,0}_{\Sc}(O_\eta)}
    \leq C_0\|\hA_{\chi_0,\eta,\s}u\|_{H^{1,0}_{\Sc}(O_\eta)}+C_0\|E_{\eta,\s}u\|_{H^{1,0}_{\Sc}(O_\eta)}\\
    &\quad\leq C_0\|\hA_{\chi_0,\eta,\s}u\|_{H^{1,0}_{\Sc}(O_\eta)}+1/2\|u\|_{L^2(U_\eta)}\implies \|u\|_{L^2(U_\eta)}\leq
2C_0\|\hA_{\chi_0,\eta,\s}u\|_{H^{1,0}_{\Sc}(O_\eta)}.
\end{align}
This implies injectivity of $\hA_{\chi_0,\eta,\s}$ on $L^2(U_\eta)$.
Using the definition of $\hA_{\chi_0,\eta,\s}$, the local X-ray transform
$f\mapsto \hI f\big|_{\widehat{\Omega}_{U_\eta}}$ is injective on
$e^{\s/\x}L^2(U_\eta)\supset L^2(U_\eta)$ for $0<\eta\leq \eta_0$. 
\end{proof}

\begin{proof}[Proof of Theorem \ref{thm:main}.]
The proof presented in Section \ref{sec:proof_in_even_case} for the even
case applies here verbatim, with the only difference that injectivity of
the $U_e$-local transform for $\hn$ on $L^2(U_e)$ now follows from
Corollary~\ref{cor:inj_pc_conn}.  
\end{proof}

\medskip

\section{Analysis of Kernels}\label{kernelanalysis}

The goal of this section is to prove Proposition~\ref{prop:mapping_diff}.
In essence, the proof proceeds as for the classical Schur criterion stating
that an operator is bounded on $L^2$ if its Schwartz kernel is uniformly
$L^1$ in each variable separately (see e.g. 
\cite[Lemma 3.7]{MR1211419}).
Hence 
it is necessary to understand well the properties of the kernels of 
${\oA}_{\chi_0,\eta,\sigma}$ and ${\hA}_{\chi_0,\eta,\sigma}$.  The
fine behavior of these kernels is perhaps best analyzed on a modified
version of Melrose's scattering blown-up space (\cite{MR1291640}), which we
describe  in Section \ref{sec:scattering_calculus}. 
We then
analyze the kernels on it in Section \ref{ssec:analysis_on_blowups}.

\subsection{The Scattering Product}
\label{sec:scattering_calculus}

We start by briefly describing blow-ups in general (for a detailed exposition see \cite{daomwk}).
Let ${Y}^d$ be compact manifold with corners and $Z$ a p-submanifold;  
this means that $Z$ is a submanifold of ${Y}$ with the property that for
each $p\in Z$ there exist coordinates for ${Y}$ of the form $(x_1,\dots,
x_k,y_1,\dots,y_{d-k})\in \o{\R}_+^k\times \R^{d-k}$ centered 
at $p$, with $x_j$ defining functions for boundary hypersurfaces of ${Y}$,
such that in terms of them $Z$ is locally expressed as the zero set of a
subset of the $x_i, $ $y_j$.  
If $Z$ is an {interior p-submanifold} of codimension at least 2, meaning
that it is locally given as $y'=(y_{j_1},\dots,y_{j_s})=0$, $ 2\leq s\leq
d-k$ in terms of such coordinates, blowing up $Z$ essentially 
amounts to introducing polar coordinates in terms of $y'$. 
Formally, let
$SN(Z)\overset{\pi}{\to} Z$ be the spherical normal bundle of $Z$ with fiber at $p\in Z$ given by
$SN_p(Z):=\left((T_p{Y}/T_pZ)\backslash\{0\}\right)/\R^+$.  It can be
shown 
that the \textit{blown up space}  
$\left[{Y};Z\right]:=SN(Z)\amalg ({Y}\backslash Z)$
admits a smooth structure as a manifold with corners such the blow down map
$\beta:\left[{Y};Z\right]\to {Y}$, given by $\beta\big|_{{Y}\backslash
  Z}=Id_{{Y}\backslash Z}$ and $\beta\big|_{SN(Z)}=\pi$, becomes smooth.
The \textit{front face} of the blown up space is given by $SN(Z)\subset [Y;Z]$.
If $Z$ is a {boundary p-submanifold}, i.e. it is contained in a boundary
hypersurface of ${Y}$, the spherical normal bundle is replaced 
by its inward pointing part,
with the rest
of the discussion unchanged.
If $P$ is a p-submanifold of ${Y}$ that intersects $Z$ with the property
 $\o{(P\backslash Z)}=P$, and $\b$ is a blow down map, then the  
\textit{lift} of $P$ is defined as $\b^*(P)=\o{\b^{-1}(P\backslash Z)}$.
If $\b=\b_1\circ\cdots \circ \b_k$ with $\b_j$ blow down maps we will write
$\b^*(P):=\b_k^*(\b_{k-1}^*(\cdots \b_1^*(P)))$.

Now let $(\oX,\p \oX)$ be a smooth compact manifold with boundary; this implies that 
$\oX^2$ is a smooth manifold with corners. 
First define the ${\bd}$-space $\oX^2_{\bd}:=\big[\oX^2;(\p \oX)^2\big]$ with blow down map $\b_1$.
We denote by $\ff_{\bd}$ the front face of this blow-up. 
If $\Delta_{\bd}:=\o{\b^{-1}_1(\Delta^\circ)}$ (the diagonal $\Delta\subset \oX^2$ is
not a p-submanifold), we let the scattering product be
$\oX^2_{\Sc}:=\big[\,\oX_{{\bd}}^2; \p(\Delta_{\bd})\big]$ 
with blow down map $\b_2:\oX^2_{ \Sc}\to \oX^2_{\bd}$. 
Set $\b_{\rm{sc}}=\b_1\circ \b_2$ and
let $\ff_{\rm{sc}}\subset \oX_{\rm{sc}}^2$ be the front face associated with $\b_2$.
We finally introduce a third blown up space obtained from $\oX^2_{\rm{sc}}$ by
blowing up the scattering diagonal ${{\Delta_{\Sc}}}:=\b_2^{*}(\Delta_{\bd})$. 
We denote the new space by $\oX_{\D}^2$ and the corresponding blow
down map by $\b_{3}$; 
let $\b_{{\D}}:=\b_{\rm{\Sc}}\circ \b_{3}$.
This space is pictured in Fig. \ref{fig:blown_up}.
By a result on commutativity of blow-ups (see \cite[Section 5.8]{daomwk}),
$\o{X}_{\D}^2$ is diffeomorphic to
$\big[[\oX_{\bd}^2;\Delta_{\bd}],\td{\b}_2^*(\p \Delta_{\bd})\big]$, where
$\td{\b}_2:[\oX_{\bd}^2;\Delta_{\bd}]\to \oX^2_{\bd}$ is the blow down
map. 
We name the various faces of $\oX_{{\D}}^2$ as in Fig. \ref{fig:blown_up}:
$\calG_{10}:=\b_{\D}^*\left(\p \oX\times \oX\right)$, 	
  $\calG_{01}:=\b_{\D}^*\left(\oX\times\p  \oX\right)$, 
	$\calG_{11}:=\b_{3}^*\left(\b_{2}^*\left(\ff_{\bd}\right)\right)$	and 	$\calG_2:=\b_{3}^*(\ff_{\Sc})$;
finally let $\calG_3$ be the front face associated with $\b_{3}$.
We will occasionally write $\calG_{\bd}$ for the collection of boundary
hypersurfaces $\{\calG_{10}, \calG_{11}, \calG_{01}\}$ and
${\calG}_{\bd}^\cup$ for their union. 
Moreover, if $p\in \p \oX$ and $O$ is a neighborhood
of $p$ in $\oX$ we let $O^2_{\D}:=\b_{{\D}}^{-1}(O^2)$.

\begin{figure}[ht]
    \begin{center}
    \includegraphics[scale=.6
      ]{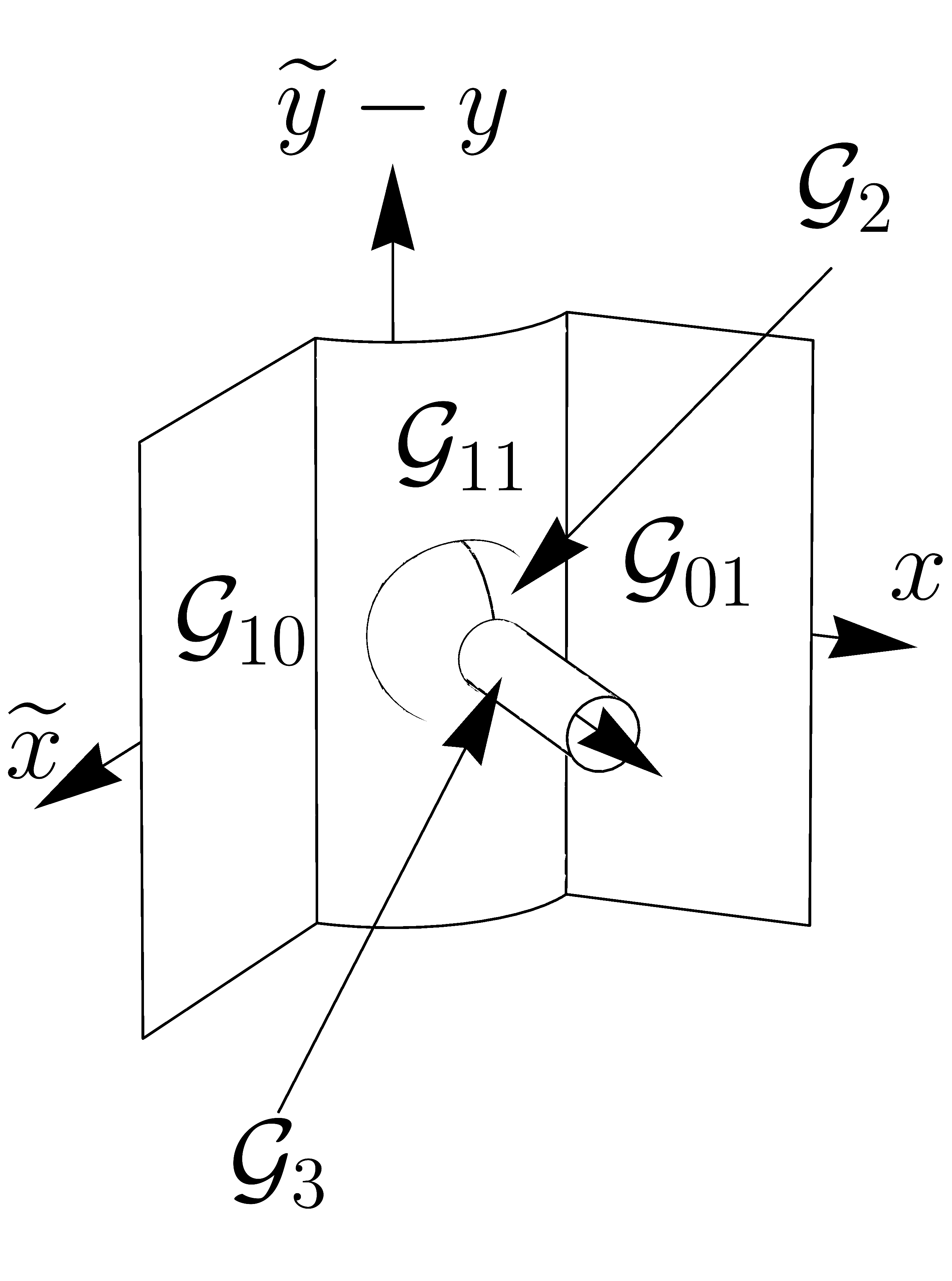}
      \caption{The modified scattering product space $\oX_{{\D}}^2$}
      \label{fig:blown_up}	
    \end{center}
      
\end{figure}

We next describe the coordinate systems we will use on $\oX^2_{\D}$.
Let $\dim \left(\oX\right)=n+1$ and $(x,y)$ and $(\td{x},\td{y})$ be two copies of the same  coordinate
system in a neighborhood $O$ of a point $p\in \p \oX$, so that 
$(x,y,\td{x},\td{y})$ is a coordinate system for $O^2\subset \oX^2$. 
Here and for the rest of this section $x$ (and thus also $\td{x}$) is a
boundary defining function for $\p \oX$. 
The projective coordinate systems $(s_1=\td{x}/x,x,y,\td{y})$ and
$(s_2=x/\td{x},\td{x},y,\td{y})$ are valid in a neighborhood of
$\calG_{01}$ and $\calG_{10}$ respectively and the coordinate functions are
smooth away from $\calG_{10}$ and $\calG_{01}$ respectively (though they
do not form coordinate systems near $\calG_2$ and $\calG_3$ in $\oX^2_{\D}$).  
In terms of the former coordinate system, $s_1$ is a defining function for
$\calG_{01}$ and $x$ a defining function for $\calG_{11}$, whereas in terms
of the latter $s_2$ is a defining function for $\calG_{10}$ and $\td{x}$ is
one for  $\calG_{11}$. 
On the other hand, either by checking directly or by using the
commutativity of the blow-up mentioned before,  one sees that a valid
coordinate system in a neighborhood of any point near 
$\calG_{11}\cap\calG_{2}$ can be obtained by appropriately
choosing $n$ of the ${\th}^j$ below,
\begin{equation}\label{eq:near_g2_cap_g11}
\left(\t=\sqrt{(s_1-1)^2+|\td{y}-y|^2},
{\th}=\frac{(s_1-1,\td{y}-y)}{\t},\s=\frac{x}{\t},y\right),
\end{equation}
where $|\cdot|$ denotes the Euclidean norm.
For instance, letting 
\begin{equation}\label{eq:U_j}
  U_J^\pm=\big\{({\th}^0,\dots,{\th}^n)\in \S^n:\pm{\th}^J>
1/\sqrt{2(n+1)}\big\},\quad  J=0,\dots, n
\end{equation}
we can cover $\S^n$ by the $U_J^\pm$ and use
${\th}^j$, $j\neq J$ as smooth coordinates on $U_\pm^J$ for each choice of
$\pm$. 
Now note that
$(X=(s_1-1)/x,Y=(\td{y}-y)/x,x,y)$ are valid smooth coordinates globally on
$(O_{\D}^2)^\circ$, and the coordinate functions are smooth up to $\calG_3$
and $\calG_2^\circ$. 
Thus one obtains a diffeomorphism $\calT $ from 
$(O^2_{{\D}})^\circ$ onto an open subset of $\o{\R^{n+1}_+}\times[0,\infty)\times\S^n$, extending
  to a smooth diffeomorphism  up to $\calG_3$ and $\calG_2^\circ$, by setting 
\begin{equation}
   \label{eq:polar_coordinates_near_front_face}
  \left({x},y,R=\sqrt{X^2+|Y|^2},{\th}=(\hat{X},\hat{Y})=\frac{(X,Y)}{R}\right)\in
  \o{\R^{n+1}_+}\times[0,\infty)\times\S^n.  
\end{equation}
Again we can choose coordinates on $\S^n$ to obtain smooth valid coordinate
systems on $(O_{\D}^2)^\circ$,  up to $\calG_3$ and $\calG_2^\circ$. 
Note that ${\th}=({\th}_0,\dots,{\th}_n)$ stands for the same functions 
in both \eqref{eq:near_g2_cap_g11} and
\eqref{eq:polar_coordinates_near_front_face} and that $R$ is a defining function for $\calG_3$. 
Moreover,
\begin{align}  %
x_{01}=&\frac{1+xR\hat{X}}{2+xR\hat{X}},\qquad %
  x_{10}=(2+xR\hat{X})^{-1},\qquad %
  x_{11}=\frac{(2+xR\hat{X})^2}{1+R}
  \label{eq:defining_functions}
\end{align}
are smooth defining functions for $\calG_{01}$, $\calG_{10}$ and
$\calG_{11}$ respectively, each  smooth up to all other boundary
hypersurfaces and non-vanishing there.

Via the diffeomorphism $\calT$, 
the expression $|dx\,dy\,dR\,d{\omega}|$ (where $d\omega$ is the volume
form on $\S^n$ induced by the round metric) pulls back to a smooth 
global section of the smooth density bundle on $(O_{{\D}}^2)^\circ$, which is
smooth and non-vanishing up to $\calG_3$ and $\calG_2^\circ$, but not up to the other boundary faces. 
The following can be shown via a straightforward computation in local
coordinates smooth up to the various boundary faces in different parts of
$O^2_{{\D\displaystyle}}$.

\begin{lemma}\label{lm:smooth_density}
Via the diffeomorphism $\calT$ defined by the
coordinates \eqref{eq:polar_coordinates_near_front_face}, the expression  
\begin{equation}\label{eq:density_scattering_product}
(R+1)^{-1}(2+xR\hat{X})^{n}|dx\, dy\, dR\,d{\omega}|
\end{equation}
pulls back to a smooth non-vanishing section of the smooth density bundle on $O^2_{{\D}}$, up to all boundary faces.
\end{lemma}

We now record the form that the lift
$\b_{{\D}}^* \td{W}$ takes in terms of
\eqref{eq:polar_coordinates_near_front_face} whenever $\td{W}\in \calV_{\Sc}(\oX)$
is identified with a vector field on $\oX^2$ acting on the left factor.
(The lift $\b_{{\D}}^* \td{W}$ is  well defined since
$\b_{\D}:(\oX^2_{\D})^\circ\to (\oX^2\setminus \Delta)^\circ$ is a diffeomorphism.) 
As before, we work in a neighborhood $O^2$ of a point $(p,p)\in \p \Delta$ 
 where we have coordinates $(x,y,\td{x},\td{y})$.
Then $\td{W}$ is spanned over $C^\infty$ by $x^2\p_x$, $x\p_{y^\a}$. 
Those lift via $\b_{1}$ to the vector fields 
$-x{s_1}\p_{s_1}+x^2\p_x$, $x\p_{y^\a}$ respectively, in coordinates
$(x,{s_1}=\td{x}/x,y,\td{y})$.
Now we lift those using $\b_2$ and find that in terms of
coordinates $(x,y,X,Y)$ they are given respectively by 
$(-1-2xX)\p_X-xY\cdot\p_Y+x^2\p_x$ and $-\p_{Y^\a}+x\p_{y^\a}$. 
Blowing up $\Delta_{\Sc}$ corresponds to using polar coordinates about $(X,Y)=0$.
Consider the sets
$U_J^\pm$,
$J=0,\dots, n$, in \eqref{eq:U_j}: 
on $U^\pm_J$ the functions ${\th}^j$, $j\neq J$, form a smooth coordinate system.
Then  for each $J$, choice of $\pm$, and $\a=1,\dots,n$,
there 
exist smooth functions $a_{J,\pm}^{j}, \,b_{J,\pm,\a}^{j}\in C^\infty(U_{J}^\pm)$ such that 
\begin{equation}
(\b_3)_* \big(\hat{X}\p_R+R^{-1}\sum_{j\neq J}
  a_{J,\pm}^{j}({\th})\p_{{\th}^{j}}\big)=\p_X\quad\text{and}\quad(\b_3)_*\big(\hat{Y}^\a\p_R+R^{-1}\sum_{j\neq
    J} b_{J,\pm,\a}^{j}({\th})\p_{{\th}^{j}}\big)=\p_{Y^\a}.   
\end{equation}
Thus  if $\td{W}\in \{x^2\p_x,x\p_y\}$ then in the set
$\{(x,y,R,{\th})\in O\times [0,\infty)\times U_{J}^\pm\}$  
 we have  $\b_{{\D}}^* \td{W}=\sum_j c^j_{J,\pm}(x,y,R,\th)\, W_j$, where $W_j$ belong to either of the two sets 
 \begin{equation}
 \calW_1= \{x^2\p_x,x\p_y,\p_R\} \quad \text{or}\quad \calW_2^J=\{R^{-1}\p_{{\th}^j},j\neq
 J\},\quad J=0,\dots, n \label{eq:W_j}
 \end{equation}
 and $c^j_{J,\pm}$ are smooth and  
 grow at most polynomially fast as $R\to \infty$.  Note also that
 $\b_{{\D}}^* \td{W}$ is smooth on $\oX_{\D}^2\setminus \calG_{3}$ and
 tangent to its boundary faces other than $\calG_3$.

\subsection{Analysis on blow-ups}
\label{ssec:analysis_on_blowups}

In this section we describe the Schwartz kernels of the operators
$A_{\chi,\eta,\s}$ defined in Section \ref{estimates} (in Lemma
\ref{lm:kernel_downstairs}) and prove two technical lemmas regarding their
regularity and dependence on the parameter $\eta$ when lifted to the
scattering stretched product space (Lemmas
\ref{lm:lift_of_kernel_general_connections} and
\ref{lm:derivative_of_lift}).  
We then use those to analyze the kernel of the difference $E_{\eta,\s}$ in
Lemma \ref{lm:kernel_difference} and finally its properties to prove
Proposition \ref{prop:mapping_diff}. 

Recall that the operators $A_{\chi,\eta,\s}$ act on 
functions supported in sets varying with the parameter $\eta$. 
As in \cite{Uhlmann2016}, it will be convenient to create an auxiliary
family of operators acting on functions defined on the same space for all
values of the parameters.  
We use the smooth one-parameter family of maps 
$\psi_\eta(\cdot)$, defined after Lemma \ref{lm:artificial_boundary} to map
diffeomorphically $\oX_\eta$ onto $\oX_0$ (locally near the boundaries).   
For $\s>0$, $\eta\geq 0$ and $\chi$ as in Section \ref{estimates} define a one-parameter family of operators by  
\begin{equation} \label{eq:kernel_A}
\td{A}_{\chi,\eta,\s}:=(\psi_{-\eta})^*\circ A_{\chi,\eta,\s}\circ (\psi_\eta)^*,
\end{equation} 
all acting on functions supported in $\oX_0$ near $p\in \p \oM_e$.
We use the notation $\td{\oA}_{\chi,\eta,\s}$ and $\td{\hA}_{\chi,\eta,\s}$
for the operators corresponding to $\n=\on$ and $\hn$.   
Similarly, for $\chi_0$ determined by Proposition \ref{prop:pdo} let 
\begin{equation}\label{eq:kernel_E}
    \td{E}_{\eta,\s}:=\td{\oA}_{\chi_0,\eta,\s}-\td{\hA}_{\chi_0,\eta,\s}.
\end{equation}
Proposition \ref{prop:mapping_diff} immediately reduces to showing the following:

\begin{manualproposition}{\ref{prop:mapping_diff}$\,\mathbf{'}$} 
\label{prop:mapping_diff_primed}
Let $\s>0$. 
Provided $O$ is a sufficiently small neighborhood of $p\in \p\oM_e$ in
$\oX_0$, for every $\d>0$  there exits $\eta_0>0$ with the property that if
$0\leq\eta< \eta_0$ one has   $\tU_\eta:=\psi_\eta(U_\eta) \subset O$ and 
\begin{equation}\label{eq:injectivity_estimate_2}
\|\td{E}_{\eta,\s}u\|_{H_{\Sc}^{1,0}(O)}
\leq \d  \|u\|_{L^2(\tU_\eta)},
\end{equation}
for all $u\in L^2(\tU_\eta)$ extended by 0 outside $\tU_\eta$.
\end{manualproposition}

We now  identify the Schwartz kernel $\k_{\td{A}_{\chi,\eta,\s}}$ of
$\td{A}_{\chi,\eta,\s}$.  It will be convenient to view it as a section of
the full smooth density bundle on $\oX_0^2$, which entails the choice of a
smooth positive density on the left $\oX_0$ factor.  This choice will not 
affect our analysis of the regularity properties of the kernel.   
We will use the the product decomposition $[0,\d_0)_{x_0}\times Y_p$ of a
  collar neighborhood of $Y_p$ in $\oX_0$  introduced in Section
  \ref{estimates} and the coordinates $y^\a$ on $Y_p$ such that the metric
  $g^0$ 
  is Euclidean in terms of $(x_0,y^1,\dots y^n)$.
Henceforth we will write $g$ for $g^0$ and $S\tM$ for its unit sphere bundle.  No
confusion will arise with the AH metric $g$, as it will not appear again.

\begin{lemma}\label{lm:kernel_downstairs}
Suppose $\n$ is a connection on $T\tM$ whose exponential map
$\exp:T\tM\to \tM$ is of class  $C^2$ and for which $\p \oM_{e}$ is
strictly convex. 
Also let $\chi\in C_c^\infty(\R)$ be even with $\chi(0)=1$, $\chi\geq 0$,
and let $\s>0$.
Then for $\eta$ sufficiently small and for $z=(x,y)$,
$\td{z}=(\td{x},\td{y})\in \oX_0$ in a sufficiently small neighborhood of
$p$, 
 we have
\begin{equation}
\begin{aligned}
  \k_{\td{A}_{\chi,\eta,\s}} =
  x^{-2}e^{-\s(1/x-1/\td{x})}2\chi\left(P(z,\td{z},\eta)\right)  
  \frac{|\det(d_{\td{z}}\,{\exp}_{z-\oeta}^{-1})(\td{z}-\oeta)|}{|{\exp}_{z-\o{\eta}}^{-1}(\td{z}-\oeta)|^n}|dzd\,\td{z}|,\\* 
  \label{eq:kernel_as_section} 
  \text{ where }
{P}(z,\td{z},\eta):=
\frac{d{x_0}\big({\exp}_{z-\o{\eta}}^{-1}(\td{z}-\o{\eta})\big)}{x\big|dy\big({\exp}_{z-\o{\eta}}^{-1}(\td{z}-\o{\eta})\big)\big|}\text{ 
  and }\o{\eta}=(\eta,0).\qquad  
\end{aligned}
\end{equation}
\end{lemma}

\begin{proof}
First examine the kernel $\k_{{A}_{\chi,\eta,\s}}$ of ${A}_{\chi,\eta,\s}$ on $\oX_\eta^2$, for fixed  $\eta\geq 0$ small.
Let $f$ be smooth and supported in a small neighborhood in $\oX_\eta$ of a point in  $U_\eta$.
We write $z'=(x',y)$, $\td{z}\,'=(\td{x}',\td{y})$ in terms of the product
decomposition $[0,\d_0)_{\x}\times Y_p$ on $\oX_\eta$ with $y$, $\td{y}$
  the coordinates on $Y_p$ as before, and also $v'=\l'\p_{x_\eta}+\omega'$
  for vectors in $T_{z'} \oX_\eta$. 
We assume throughout that $x',$ $|y|$ and $\eta$ are sufficiently small
that the conclusions of Lemma \ref{lm:geod_cond} are true for  all
geodesics entering the computation of ${A}_{\chi,\eta,\s}f({{z'}}) $ (see
the discussion following Lemma \ref{lm:geod_cond}). 
 Writing $d\l_g$ for the measure induced by $g$ on the fibers of $T\tM$, compute 
\begin{align}  
{A}_{\chi,\eta,\s}f({{z'}})=&{x'}^{-2}
e^{-\s/x'}\int_{S_{{{z'}}}\oX_\eta}\chi\!\left(\frac{\l'}{{x'}|\w'|}\right)\int_{-\infty}^\infty
(e^{\s/x_\eta}f)\big|_{\td{z}\,'={\exp}_{{{z'}}}(tv')}dt\,d{\mu}_{{g}}\\ 
  =&{x'}^{-2}
  e^{-\s/x'}\int_{S_{{{z'}}}\oX_\eta}2\chi\!\left(\frac{\l'}{{x'}|\w'|}\right)\int_{0}^\infty
  (e^{\s/x_\eta}f)\big|_{\td{z}\,'={\exp}_{{{z'}}}(tv')}dt\,d{\mu}_{{g}}\\%
  =&{x'}^{-2}
  e^{-\s/x'}\int_{T_{{{z'}}}\oX_\eta}2\chi\!\left(\frac{\l'}{{x}'|\w'|}\right)
  (e^{\s/x_\eta}f)\big|_{\td{z}\,'={\exp}_{{{z'}}}(v')}
  \frac{d{\l}_{{g}}}{|v'|^n}\\%
    =&{x'}^{-2} e^{-\s/x'}%
    \int_{\oX_\eta}2\chi\!\left(\frac{d{x}_\eta\!\left({\exp}_{{{z'}}}^{-1}({\td{z}\,'})\right)}{{{x'}}|dy\!\left({\exp}_{{{z'}}}^{-1}({\td{z}\,'})\right)|}\right)   
  \frac{e^{\s/\td{x}'}f({\td{z}\,'})}{|{\exp}_{{{z'}}}^{-1}({\td{z}\,'})|^n}(\exp_{{z'}})_*(d{\l}_{{g}}). 
  \label{eq:kernel_of_A}
\end{align}
By Lemma \ref{lm:geod_cond}  the two integrals with respect to $t$ above
are in fact over finite intervals $(-\d_1,\d_1)$ and $[0,\d_1)$,
  respectively. 
Moreover, $d\l_g(v')=\sqrt{ \det g(z')}|dv'|=|dv'|$ in terms of fiber coordinates, since $g$ is Euclidean near $p$.
Thus we can take 
\begin{equation}
  \k_{{A}_{\chi,\eta,\s}} =
  {x_\eta^{-2}(z')}e^{-\big(\frac{\s}{x_\eta(z')}-\frac{\s}{x_\eta(\td{z}')}\big)}
  2\chi\!\left(\frac{d{x_\eta}\!\left({\exp}_{z'}^{-1}(\td{z}\,')\right)}{x_\eta(z')\big|dy\!\left({\exp}_{z'}^{-1}(\td{z}\,')\right)\big|}\right)  
  \frac{|\det(d_{\td{z}\,'}\,{\exp}_{z'}^{-1})(\td{z}\,')|}{|{\exp}_{z'}^{-1}(\td{z}\,')|^n}|dz'd\,\td{z}\,'|.
\end{equation}
Conjugation by $\psi_{\eta}$ in \eqref{eq:kernel_A} corresponds to
replacing $(z',\td{z}\,')$ by $(z-\oeta,\td{z}-\oeta)$ in the Schwartz
kernel of $A_{\chi,\eta,\s}$, where $z$, $\td{z}$ are expressed in terms of
the product decomposition $[0,\d_0)_{x_0}\times Y_p$ on $\oX_0$. 
Noting that $dx_\eta=dx_0$ completes the proof.
\end{proof}

In the next two lemmas we use \eqref{eq:kernel_as_section} to analyze the
Schwartz kernel of $\td{A}_{\chi,\eta,\s}$ on $\big(\o 
{X}_0\big)^2_{{\D}}$ near $\b_{\D}^{-1}(p,p)$.
Since the proof of Proposition \ref{prop:mapping_diff} has been reduced to
showing Proposition \ref{prop:mapping_diff_primed}, 
from now on the entire analysis will be on $\oX_0$.
We will thus drop the subscript and write $\oX$ to mean $\oX_0$.
We write $z=(z^0,z^\a)=(x,y^\a)$ and
$\td{z}=(\td{z}\,^0,\td{z}\,^\a)=(\td{x},\td{y}^\a)$ for points in the left
and right factor of $\oX$ respectively with respect to the product
decomposition $[0,\d_0)_{x_0}\times Y_p$. 
Denote by $\nu$ a fixed smooth non-vanishing section of
$\Omega\big({\oX}^2_{{\D}}\big)$, the smooth density bundle on
$\o{X}^2_{{\D}}$; also recall the  notations $\calG_*$ introduced in
Section \ref{sec:scattering_calculus} for the various boundary faces of $\o 
{X}^2_{{\D}}$.
 In what follows, whenever we say that a function $f$ vanishes to infinite
 order at a collection $\{\calF_j\}_{j=1}^J$ of boundary hypersurfaces of
 a manifold with corners, we mean that if 
 $x_j$ is a defining function of $\calF_j$ then for any $(N_1,\dots,N_J)\in
 \mathbb{N}_0^J$ one has $\prod_{j=1}^J x_j^{-N_j}f\in L^\infty $ (thus
 this is purely a statement regarding 
 the growth of $f$ without any mention of the behavior of its derivatives near $\calF_j$).

\begin{lemma}\label{lm:lift_of_kernel_general_connections}
Let the hypotheses of Lemma \ref{lm:kernel_downstairs} hold.
For a sufficiently small neighborhood $O$ of $p$ in $\oX$ there exists
$\eta_0>0$ depending on $O$, $\n$ and $\chi$ such that 
\begin{equation}\label{eq:lift_of_kernel}
  \b_{\D}^*(\k_{\td{A}_{\chi,\eta,\s}})=K_{\n}(\cdot,\eta)\cdot\nu,\text{ where }
  K_\n(\cdot,\cdot)\in
  C^0\big(O_{\D}^2\times[0,\eta_0)\big).
\end{equation} 
 Moreover, $K_\n$ is $C^1$ away from $\calG_2\times [0,\eta_0)$ and ${\calG}_{\bd}^\cup\times [0,\eta_0)$,
 it vanishes to infinite order on $\calG_{\bd}\times
 [0,\eta_0)$,\footnote{With some abuse of notation, this means on
     $\calG\times [0,\eta_0)$ for $\calG\in \calG_{\bd}$.} and its
     restriction to $\calG_3\times [0,\eta_0)$ is independent of $\n$.  
\end{lemma}

\begin{proof}

Throughout this proof we always assume that we are working in a small
enough neighborhood $O^2$ and with small enough $\eta_0$ that the
coordinates $(x,y,\td{x},\td{y})$ are valid, $g$ is Euclidean on $O$,
$\exp_{z-\oeta}^{-1}(\td{z}-\oeta)$ is a $C^2$ diffeomorphism onto its
image for $(z,\td{z})\in O^2$ and $0\leq \eta < \eta_0$, and the conclusion
of Lemma \ref{lm:geod_cond} holds for geodesics entering the computation of
$A_{\chi,\eta}$ for such $\eta$. 

Before we lift \eqref{eq:kernel_as_section} to  $\oX^2_{{\D}}$ to study its
regularity,  we analyze its various factors on $\oX^2$. 
The main difficulty in proving Lemmas
\ref{lm:lift_of_kernel_general_connections} and \ref{lm:derivative_of_lift}
is that whenever Taylor's Theorem is used to identify the leading order
behavior of a $C^k$ function at a point,  the remainder term is generally
not $C^{k}$. 
To circumvent this issue in our case, we use Taylor's Theorem for the function $t\mapsto
{\exp}_{z-\o{\eta}}^{-1}(z-\o{\eta}+t(\td{z}-z))$ to write two different
expressions for $\exp^{-1}_{z-\o{\eta}}(\td{z}-\o{\eta})$, each one of
which will be used in different parts of the argument: 
\begin{align}
dz^k({\exp}_{z-\oeta}^{-1}(\td{z}-\oeta))
=&{p}_j^k(z,\td{z},\eta)(\td{z}-z)^j\label{eq:taylor_first_order}\\* 
  =&(\td{z}-z)^k+{p}_{ij}^k(z,\td{z},\eta)(\td{z}-z)^i(\td{z}-z)^j,\quad
  \text{ where }\label{eq:Taylor_2nd_order} \\
{p}_j^k(z,\td{z},\eta):=&\int_0^1\p_{\td{z}^j}\big(dz^k{\exp}_{z-\o{\eta}}^{-1}\big)\big|_{{z}-\o{\eta}+\t(\td{z}-z)}d\t\in 
C^{1}\big(O^2\times [0,\eta_0)\big),\\  
{p}_{ij}^k(z,\td{z},\eta):=&
\int_0^1(1-\t)\p_{\td{z}^i\td{z}^j}\big(dz^k{\exp}_{z-\o{\eta}}^{-1}\big)\big|_{{z}-\o{\eta}+\t(\td{z}-z)}d\t\in
C^{0}\big(O^2\times [0,\eta_0)\big),  
\end{align}
{ with}
\begin{align}
\begin{tabular}{c c}
 $p^k_j(z,z,\eta)=\d_j^k$ \text{ and}&  ${p}_{ij}^k(z,z,\eta)=\frac{1}{2}{{\G}}_{ij}^k(z-\o{\eta})$.
\end{tabular} %
\end{align}
Here ${{\G}}_{ij}^k$ denote the connection coefficients of ${\n}$ in coordinates $(x,y)$. 
Now \eqref{eq:taylor_first_order} and \eqref{eq:Taylor_2nd_order} can be
used to show regularity of the factors of \eqref{eq:kernel_as_section}. 
By \eqref{eq:taylor_first_order}, 
\begin{equation}
  \begin{aligned}
  |{\exp}_{z-\o{\eta}}^{-1}&(\td{z}-\o{\eta})|^2={G}_{ij}(z,\td{z},\eta)(\td{z}-z)^i(\td{z}-z)^j,
  \text{ where } 
  	\;{G}_{ij}\in C^{1}\big(O^2\times [0,\eta_0)\big),\\*
  	\;
            &{G}_{ij}(z,{z},\eta)={\d}_{ij},\quad G_{ij} \text{
              positive definite in } O^2\times [0,\eta_0). \label{eq:distance}
  \end{aligned}
\end{equation}
To analyze $P$ from \eqref{eq:kernel_as_section} write, using
\eqref{eq:taylor_first_order} and \eqref{eq:Taylor_2nd_order},  
\begin{align}\label{eq:ratio1}
  {P}(z,\td{z},\eta)=&\dfrac{{p}_j^0(z,\td{z},\eta)(\td{z}-z)^j}{x\big({q}_{ij}(z,\td{z},\eta)(\td{z}-z)^i(\td{z}-z)^j\big)^{1/2}},\\    
  =&\dfrac{\td{x}-x+{p}_{ij}^0(z,\td{z},\eta)(\td{z}-z)^i(\td{z}-z)^j}{x\big(|\td{y}-y|^2+{q}_{ijk}(z,\td{z},\eta)(\td{z}-z)^i(\td{z}-z)^j(\td{z}-z)^k\big)^{1/2}},   \label{eq:ratio2}
\end{align}
where ${q}_{ij}={\d}_{\a\b}p^\a_ip^\b_j\in
C^{1}\big(O^2\times[0,\eta_0)\big)$, ${q}_{ijk}\in
  C^{0}\big(O^2\times[0,\eta_0)\big)$. 
We finally have
\begin{equation}\label{eq:Jacobian}
|\det(d_{\td{z}}\,{\exp}_{z-\o{\eta}}^{-1})(\td{z}-\o{\eta})|\in
C^1\big(O^2\times [0,\eta_0)\big), \quad
  |\det(d_{\td{z}}\,{\exp}_{z-\o{\eta}}^{-1})({z-\o{\eta}})|=1.
\end{equation} 
\smallskip

We now lift the various factors of the kernel.
As explained in Section \ref{sec:scattering_calculus}, near any point
in $(O_{{\D}}^2)^\circ$ we obtain a smooth coordinate system with a
suitable choice of $n$ of the $\th^j$ in
\eqref{eq:polar_coordinates_near_front_face}. 
Moreover, the functions $({x},y,R,{\th})$ are smooth up to
$\calG_2^\circ$ and $\calG_3$, and $x$, $R$ are defining
functions for $\calG_2$ and $\calG_3$ respectively. 
Since $\b_{\D}$ is smooth, \eqref{eq:Jacobian} implies that
\begin{equation}
	\b_{{\D}}^*\big(|\det(d_{\td{z}}\,{\exp}_{z-\o{\eta}}^{-1})(\td{z}-\o{\eta})|\big)\in
        C^1\big(O^{2}_{{\D}}\times[0,\eta_0)\big) 
\end{equation} 
and it is identically 1 at $\calG_2$ and $\calG_3$. 
Now write
$\hat{Z}=(x \hat{X},\hat{Y})$, so that $\td{z}-z=xR\hat{Z}$;  by \eqref{eq:distance},
\begin{align}
\b_{{{\D}}}^*|{\exp}_{z-\oeta}^{-1}(\td{z}-\oeta)|^{-n}=&
x^{-n}R^{-n}\big({G}_{ij}(z,z+xR\hat{Z},\eta)\hat{Z}^i\hat{Z}^j\big)^{-n/2}\\ 
=&
x^{-n}R^{-n}\big({G}_{\a\b}\hat{Y}^\a\hat{Y}^\b+2x{G}_{0\b}\hat{X}\hat{Y}^\b+x^2{G}_{00}\hat{X}^2\big)^{-n/2}.\label{eq:diagonal}   
\end{align}
Next pull back $\chi(P)$, writing it in two ways using
\eqref{eq:ratio1} and \eqref{eq:ratio2}: 
\begin{align}
  \b_{{{\D}}}^*(\chi_{}({P}))=&\chi\!\left(\frac{{p}^0_{j}(z,z+xR\hat{Z},\eta)\hat{Z}^j}
  {x\big(q_{ij}(z,z+xR\hat{Z},\eta)\hat{Z}^i\hat{Z}^j\big)^{1/2}}\right)\label{eq:lift_c1_version}\\%
    =&\chi\!\left(\dfrac{\hat{X}+R
      \;{p}_{\a\b}^0{}\hat{Y}^\a\hat{Y}^\b+xR\big(\;2{p}_{0\b}^0{}\hat{X}\hat{Y}^\b+x{p}_{00}^0{}\hat{X}^2\big)}{\big(|\hat{Y}|^2+x  
      R\;{q}_{ijk}{}\hat{Z}^i\hat{Z}^j\hat{Z}^k\big)^{1/2}}\right)\label{eq:lift_polar}, 
\end{align}
where in \eqref{eq:lift_polar} the $p^k_{ij}$ and $q_{ijk}$ are all evaluated at $(z,z+xR\hat{Z},\eta)$.
Some caution is required when the denominator of $P$ approaches 0.
Near any point in $(O_{\D}^2)^\circ\times [0,\eta_0)$  the expression 
$\b_{{{\D}}}^*(\chi_{}({P}))$ is $C^2$. This is because any such point projects
via $\b_{{\D}}$ to a pair of points away from the diagonal, which implies
that if the denominator of $P$ in \eqref{eq:kernel_as_section} vanishes the
numerator does not. Thus $\chi(P)=0$ there,  since $\chi$ is compactly
supported.  Now suppose we are given
$q'=(x',y',R',{\th}',\eta')\in(\calG_2^\circ\cup\calG_3)\times 
[0,\eta_0)$, so either $x'=0$ or $R'=0$. 
Since $|{\th}|=|(\hat{X},\hat{Y})|=1$, either $\hat{X} $ or $|\hat{Y}|$ are bounded away from 0.
If $|\hat{Y}|\leq \e$ for some $\e>0$, the numerator
of $P$  is bounded below in absolute value by ${\sqrt{1-\e^2}-CR(\e+x)}$, therefore if $\e$
is small enough  the numerator is bounded
below by a positive constant in a sufficiently small neighborhood of $q'$. 
This again implies that $\chi(P)$ is continuous at $q'$ in this case.
On the other hand, if $|\hat{Y}|\geq \e$ then in a neighborhood of $q'$
the denominator is bounded away from 0. 
We conclude that $\b_{{{\D}}}^*(\chi_{}({P}))$ extends continuously to
$(O_{{\D}}^2\setminus {\calG}_b^\cup)\times [0,\eta_0)$ and, in fact, it is
  $C^1$ away from $\calG_2\times [0,\eta_0)$ and ${\calG}_{\bd}^\cup\times
    [0,\eta_0)$  
 due to \eqref{eq:lift_c1_version}.
A similar analysis applies to show that $R^nx^n
\b_{{{\D}}}^*\big(|{\exp}_{z-\oeta}^{-1}(\td{z}-\oeta)|^{-n}\big)\in
C^1\big((O_{{\D}}^2\setminus{\calG}_{\bd}^\cup)\times
[0,\eta_0)\big)$ in the support of $\b_{{{\D}}}^*(\chi_{}({P}))$.

Next we have 
\begin{equation}
\b_{{{\D}}}^*(x^{-2}e^{-\s/x+\s/\td{x}})=x^{-2}e^{-\s\frac{R\hat{X}}{1+xR\hat{X}}}
\quad\text {and}\quad
\b_{{{\D}}}^*|dx\,dy\,d\td{x}\,d\td{y}|=x^{n+2}R^{n}|dx\,dy\,dR
\,d{\omega}|, 
\end{equation}
so upon combining the lifts of the factors in \eqref{eq:kernel_as_section}
and using Lemma \ref{lm:smooth_density} we find that
$\b_{\D}^*(\k_{\td{A}^\n_{\chi,\eta,\s}})= K_{\n}\cdot\nu$, where, up to  a
smooth non-vanishing multiple depending on $\nu$, 
\begin{equation}
  \begin{aligned}
 K_{\n}=2
     e^{-\frac{\s R\hat{X}}{1+xR\hat{X}}}\chi_{}\big(P(z,z+xR\hat{Z},\eta)\big)  %
  \frac{\big|\det(d_{\td{z}}\,{\exp}_{z-\oeta}^{-1})(z-\oeta+xR\hat{Z})\big|}{\big({G}_{ij}(z,z+xR\hat{Z},\eta)\hat{Z}^i\hat{Z}^j\big)^{n/2}} 
  \frac{(R+1)}{(2+xR\hat{X})^{n}}. \label{eq:full_expression_for_lift_of_kernel}
  \end{aligned}
\end{equation}
By our analysis of the various factors we conclude that
$K_\n$ is $C^1$ on $O_{\D}^2$ away from $\calG_2\times [0,\eta_0)$ and ${\calG}_b^\cup\times [0,\eta_0)$,
 and continuous up to $\calG_2^\circ\times [0,\eta_0)$.
Thus Taylor's theorem in terms of $R$ applies for $x>0$; we find that for $R\geq0$ small and $x>0$ 
 \begin{equation}\label{eq:kernel_r=0}
 	K_\n=2^{-n+1}\chi\!\left(\frac{\hat{X}}{|\hat{Y}|}\right)
  |\hat{Z}|^{-n}+R \,  \Lambda_\n(x,y,R,{\th},\eta), 
 \end{equation}
 where $\Lambda_\n$ is continuous in all of its arguments, up to $x=0$:
 observe that by \eqref{eq:lift_c1_version}  
 $\b_{{{\D}}}^*(\chi_{}({P}))=\chi\!\left(\frac{a_1(z,xR\hat{Z},\eta,\hat{Z})}{x\;(a_2(z,xR\hat{Z},\eta,\hat{Z}))^{1/2}}\right)$ 
 with $a_j$ both $C^1$ in their arguments, so upon taking an $R$-derivative
 the chain rule generates a factor of $x$ which cancels the one in the
 denominator of the argument of $\chi.$ 
From \eqref{eq:kernel_r=0} we conclude $K_\n\big|_{\calG_3}$ is indeed independent of $\n$.

Finally the vanishing of $K_\n $ to infinite order on $\calG_{\bd}\times
[0,\eta_0)$ follows as in the proof of \cite[Proposition 3.3]{Uhlmann2016}
  (also see \cite[p.45]{EptaminitakisNikolaos2020Gxto}), where it is shown
  that $e^{-\frac{\s R\hat{X}}{1+xR\hat{X}}}\chi_{}\big(P\big)$ decays
  exponentially (or vanishes identically) as $R\to \infty$, and upon taking
  into account that all other factors of the kernel grow at most
  polynomially fast as $R\to \infty$, uniformly in $\eta$. 
\end{proof}

\begin{lemma}\label{lm:derivative_of_lift}
Let the hypotheses and notations of Lemma \ref{lm:lift_of_kernel_general_connections} be in effect.
Also let ${W}$ be the lift to $\oX^2_{{\D}}$ of a vector field in
$\calV_{\Sc}(\oX)$ acting on the left factor of $\oX^2$ and  $x_3$ be  a
defining function for $\calG_3$, smooth and non-vanishing on
$\oX^2_{{\D}}\setminus \calG_3$.  
 Then for any sufficiently small neighborhood $O$ of $p$ in $\oX$ there exists $\eta_0>0$ such that 
\begin{equation}
	x_3 W ( K_{\n})=K_{\n,W}(\cdot,\cdot)\in
        C^0\big(O_{\D}^2\times
        [0,\eta_0)\big),\label{eq:continuity of derivative of kernel} 
\end{equation} vanishing to infinite order on $\calG_{\bd}\times[0,\eta_0)$.
Moreover, in terms of a product decomposition $\calG_3\times
[0,\e)_{x_3}\times [0,\eta_0)_\eta$ for $\oX_{\D}^2\times [0,\eta_0)_\eta$
      near $\calG_3\times [0,\eta_0)_\eta$  one has 
\begin{equation}\label{eq:lemma_kernel_expression}
	x_3W ( K_{\n})=\k_{W}(q,\eta)+x_3\k_{\n,W}(q,x_3,\eta),\quad q\in \calG_3
\end{equation}
where $\k_{W}\in C^0\big(\calG_3\times [0,\eta_0)\big)$ is independent
  of $\n$ and $\k_{\n,W}\in C^0\big(\calG_3\times [0,\e)\times[0,\eta_0)\big)$.
\end{lemma}

\begin{remark}
	The kernel $K_\n$ is well defined only up to a non-vanishing smooth
        multiple, since there isn't a canonical 
        non-vanishing smooth density on $\oX^2_{\D}$. 
	However, by  the comments at the end of Section
        \ref{sec:scattering_calculus}, $x_3 W$ is smooth  on $\oX_{\D}^2$,
        hence by Lemma \ref{lm:lift_of_kernel_general_connections} and
        \eqref{eq:kernel_r=0} it follows that multiplying $K_\n$ by a
        function smooth on $\oX_{\D}^2$ does not affect the result. 

\end{remark}

\begin{remark}
	The fact that the leading order term of $x_3W ( K_{\n})$ at
        $\calG_3\times [0,\eta_0)$ in \eqref{eq:lemma_kernel_expression} is
          independent of $\n$ is expected, since that was the case for
          $K_{\n}$, and $x_3 W$ is tangent to $\calG_3$. 
\end{remark}

\begin{proof}

Recall the diffeomorphism $\calT$ from Section
\ref{sec:scattering_calculus} defined on $O_{\D}^2\setminus
\calG_{\bd}^\cup$ for a small neighborhood $O$ of $p$, 
and  let 
$\calU_J^{\pm}:=\calT^{-1} \big(\o{\R_{+}^{n+1}}\times 
[0,\infty)\times U_J^\pm\big)\times [0,\eta_0)$ for $\eta_0>0$ ,with $U_J^\pm$ as in \eqref{eq:U_j}. 
Then $\bigcup_{J,\pm}\calU_J^{\pm}$  covers $(O^2_{\D}\setminus \calG_{\bd}^\cup)\times [0,\eta_0)$,
 and in each of the $\calU_J^{\pm}$ we have valid coordinates $(x,y,R,\th^j,\eta)$, $j\neq J$.
By the remarks at the end of Section \ref{sec:scattering_calculus}, it
suffices to show the claim on $\calU_J^{\pm}$ for $J=0,\dots,n$ assuming
that $W=W_1, W_2$, where 
$W_1\in \calW_1$, $W_2\in\calW_2^J$ (see \eqref{eq:W_j}),
and 
use a partition of unity subordinate to the cover $\{\calU^\pm_J\}_{J,\pm}$ to obtain the statement for general $W$.

We will use the expression \eqref{eq:full_expression_for_lift_of_kernel} we
computed for $K_\n$ in Lemma \ref{lm:lift_of_kernel_general_connections}. 
Suppose first that  $W_1\in\calW_1$: then $W_1$ is smooth for $x\geq 0$ and we
will show continuity of $W_1K_\n$ up to $x=0$ (so in this case the leading
order term at $\calG_3$ in \eqref{eq:lemma_kernel_expression} vanishes).  
Recall the notation $\hat{Z}=(x\hat{X},\hat{Y})$ and observe that 
\begin{equation}\label{eq:smooth_terms}
 {
e^{-\s\frac{R\hat{X}}{1+xR\hat{X}}}\big|\det(d_{\td{z}}\,{\exp}_{z-\oeta}^{-1})(z-\oeta+xR\hat{Z})\big|}\frac{(R+1)}{(2+xR\hat{X})^{n}}=\Lambda_0(z,R,R{\th},\eta),  
\end{equation}
with $\Lambda_0(z,R,v,\eta)\in C^1\big(O\times
  [0,\infty)\times\R^{n+1}\times[0,\eta_0)\big)$ for some small neighborhood $O$ of $p$ and small $\eta_0>0$. 
Therefore, $W_1\Lambda_0$ is continuous on the same space. 
Note that $\Lambda_0(z,0,0,\eta)$ is independent of $\n$.
Using \eqref{eq:lift_c1_version}, we see that $W_1\big(\chi_{}\big(
P(z,z+xR\hat{Z},\eta)\big)\big)$ is continuous up to $\calG_2^\circ$ and
$\calG_3$ and, 
similarly to the proof of Lemma \ref{lm:lift_of_kernel_general_connections},
$W_1\big({G}_{ij}(z,z+xR\hat{Z},\eta)\hat{Z}^i\hat{Z}^j\big)^{-n/2}$ is
continuous in the support of $\chi_{}\big(P\big)$ 
and $\chi'\big(P\big)$.
Since $K_\n=\Lambda_0 \chi(P) (G_{ij}\hat{Z}^i\hat{Z}^j)^{-n/2}$,
the product rule implies the continuity of $W_1\, K_\n$ away from
$\calG_{\bd}^\cup$. 
Again by the proof of \cite[Proposition 3.3]{Uhlmann2016}, 
 $W_1K_\n$ vanishes to infinite order at $\calG_{\bd}$ uniformly in $\eta $, and thus
$W_1K_\n\in C^0\big(O_{{\D}}^2\times[0,\eta_0)\big)$.
Upon multiplying by $x_3$ throughout, we obtain the claim for $W_1\in \calW_1$, with $\k_{W_1}\equiv0$ in
\eqref{eq:lemma_kernel_expression}.

Now fix a $J$ and suppose  $W_2\in \calW_2^J$, so that $ W_2=R^{-1}\p_{\th^j}$ for some $j\neq J$.
We will analyze $W_2 K_{\n}$, again looking away from  $\calG_{\bd}^\cup$ first.
By \eqref{eq:smooth_terms} and the chain rule we have that 
\begin{equation}\label{eq:first_factor}
\p_{{\th}^j}\Lambda_0=R\sum_{m=0}^n\p_{v^m}\Lambda_0\,\p_{{\th}^j}\th^m	
\end{equation}
is continuous up to $x=0$ on $\calU_J^{\pm}$.

For $\p_{{\th}^j}(\b_{{{\D}}}^*(\chi_{}({P})))$, as noted in the proof of
Lemma \ref{lm:lift_of_kernel_general_connections},  
$$\b_{{{\D}}}^*(\chi_{}({P}))=\chi_{}\left(\frac{a_1(z,xR\hat{Z},\eta,\hat{Z})}{x(a_2(z,xR\hat{Z},\eta,\hat{Z}))^{1/2}}\right),$$  
where $a_j(z,u,\eta,v)$ is $C^1$ in $(z,u,\eta)$ and $C^\infty$ in $v$.
Thus in $\calU_J^{\pm}$, for $x, R>0$ and $j\neq J$,
\begin{align}
	\p_{{\th}^j}(\b_{{{\D}}}^*(\chi_{}({P})))=
	&\chi'(P)\Big(R\,\p_u\big(a_1/a_2^{1/2}\big)(z,xR\hat{Z},\eta,\hat{Z})\cdot \p_{{\th}^j}\hat{Z}\\*
	&\qquad\qquad+x^{-1}\p_{v}(a_1/a_2^{1/2})(z,xR\hat{Z},\eta,\hat{Z})\cdot\p_{{\th}^j}\hat{Z}\Big).
\end{align}
Now use Taylor's Theorem for the function $R\mapsto
\p_{v}\big(a_1/a_2^{1/2}\big)(z,xR\hat{Z},\eta,\hat{Z})\cdot\p_{{\th}^j}\hat{Z}$
(which is $C^1$ in the support of $\chi'(P)$) for 
$x>0$, and the fact that $\p_{v^j}\big(a_1/a_2^{1/2}\big)\big|_{u=0}=\dfrac{\d_j^0(\d_{\a\b}v^\a
  v^\b)-v^0 \d_{j\a}v^\a}{(\d_{\a\b}v^\a v^\b)^{3/2}}$ %
 to find that
 \begin{equation}
   \begin{aligned}
  &\p_{{\th}^j}(\b_{{{\D}}}^*(\chi_{}({P})))
  =\chi'(P)\bigg(R\p_u\big(a_1/a_2^{1/2}\big)(z,xR\hat{Z},\eta,\hat{Z})\cdot {\color{black}\p_{\th^j}\hat{Z}}\\
  &\quad+x^{-1}\frac{\d_m^0(\d_{\a\b}v^\a v^\b)-v^0
          \d_{m\a}v^\a}{(\d_{\a\b}v^\a
          v^\b)^{3/2}}\Big|_{v=\hat{Z}}\p_{{\th}^j}\hat{Z}^m+R\,
        b_{k\ell}(z,xR\hat{Z},\eta,\hat{Z})\hat{Z}^k\p_{{\th}^j}\hat{Z}^\ell\bigg);\label{eq:near_diagonal}  
\end{aligned} 
 \end{equation}
here $b_{k\ell}(z,u,\eta,v)$ is $C^0$ in $(z,u,\eta)$ and $C^\infty$ in $v$.
Note that on $\calU_J^\pm$ and for $j\neq J$
\begin{equation}\label{eq:cases}
	\p_{{\th}^j}\hat{Z}^m=\begin{cases}
	x^{\d_{0m}}\d_j^m, & m\neq J\\
-x^{\d_{0m}}{\th}^j/{\th}^m, & m=J
	\end{cases},
\end{equation}
so in particular $\p_{{\th}^j}\hat{Z}^m$ is smooth on $\calU_J^\pm$.
Therefore, evaluating at $v=\hat{Z}$ in \eqref{eq:near_diagonal}
we obtain
\begin{equation}
	\p_{{\th}^j}(\b_{{{\D}}}^*(\chi_{}({P})))=\chi'(P)\frac{(\d_{\a\b}\hat{Y}^\a
          \hat{Y}^\b)\p_{{\th}^j}\hat{X}-\hat{X}
          \d_{m\a}\hat{Z}^\a\p_{{\th}^j}\hat{Z}^m}{(\d_{{\a}{\b}}\hat{Y}^\a\hat{Y}^\b)^{3/2}}
	+R\chi'(P)\Lambda_1,
	\label{eq:second_factor}
\end{equation}
where $\Lambda_1\in C^0(\calU_J^{\pm})$ in the support of $\chi'(P)$ (as in the proof of Lemma
\ref{lm:lift_of_kernel_general_connections}) and bounded as $R\to\infty$.

We similarly  compute that 
\begin{align} %
	\p_{{\th}^j}\big({G}_{k
          \ell}(z,z+xR\hat{Z},\eta)\hat{Z}^k\hat{Z}^\ell\big)^{-\frac{n}{2}}=-\frac{n}{2}|\hat{Z}|^{-n-2}\big(2\d_{k\ell}\hat{Z}^k\p_{{\th}^j}\hat{Z}^\ell\big)+R\Lambda_2  
          \label{eq:third_factor}
\end{align}
with $\Lambda_2\in C^0(\calU_J^{\pm})$ and bounded as $R\to \infty$ in the support of $\chi(P)$.

Now apply $\p_{\th^j}$ to $K_\n=\Lambda_0 \chi(P) (G_{ij}\hat{Z}^i\hat{Z}^j)^{-n/2}$ and use the product rule.
Using \eqref{eq:first_factor}, \eqref{eq:second_factor} and
\eqref{eq:third_factor}, together with \eqref{eq:smooth_terms} and the
remarks following \eqref{eq:kernel_r=0} to deal with the non-differentiated
factors, we obtain \eqref{eq:lemma_kernel_expression} in $\calU_J^{\pm}$
for $x_3W=\p_{\th^j}$. 
Again by the proof of Proposition 3.3 in \cite{Uhlmann2016}, $W_2K_\n$
decays exponentially fast or vanishes identically as $R\to \infty$  on
$\calU^\pm_J$, uniformly in $\eta$, and we are done. 
\end{proof}

We have shown the regularity results we need for the kernel of $\td{A}_{\chi,\eta,\s}$,
under hypotheses which apply for both $\n=\hn,\o{\n}$. 
We now analyze the lift of the kernel $\k_{\td{E}_{\eta,\s}}$ of
$\td{E}_{\eta,\s}$ (viewed as a section of the smooth density bundle on
$\oX^2$, as usual):

\begin{lemma}\label{lm:kernel_difference}
Let $W$ and $x_3$  be as in the statements of Lemmas
\ref{lm:lift_of_kernel_general_connections} and
\ref{lm:derivative_of_lift}. 
Then for any sufficiently small neighborhood $O$ of $p$ in $\oX$ there
exists $\eta_0>0$ such that upon writing  $
\b^*_{{\D\displaystyle}}(\k_{\td{E}_{\eta,\s}})=K_E\cdot\nu$ one has
$x_3^{-1} 
 K_E,$ $W K_E\in C^0\big (O_{\D\displaystyle}^2\times [0,\eta_0)\big)$ and they both vanish
  to infinite order on $\calG_{\bd}\times [0,\eta_0)$. 
Moreover, both $W K_E$ and $x_3^{-1}
 K_E$ vanish identically for $\eta=0$.
\end{lemma}

\begin{proof}
First observe that Lemmas \ref{lm:cnvx_bdry} and \ref{lm:regularity} imply
that for
$\s>0$ and $\chi_0$ fixed in Proposition \ref{prop:pdo}, Lemmas
\ref{lm:kernel_downstairs}, \ref{lm:lift_of_kernel_general_connections} and 
\ref{lm:derivative_of_lift} apply to both $\on$ and $\hn$, provided
$\eta_0$ and $O$ are sufficiently small: 
note that one needs $O$ to be small enough that if $\supp(\chi_0)\subset
    [-M,M]$ then $M x\leq \min\{\td{C}_{\hn},\td{C}_{\o{\n}}\}\sqrt{x}$ in
    $O$, where $\td{C}_{\hn}$, $\td{C}_{\o{\n}}$ are the constants of Lemma 
    \ref{lm:geod_cond} corresponding to the two connections. 
Now we observe that
$W K_E$ and $x_3^{-1}
 K_E\in C^0\big (O_{\D\displaystyle}^2\times [0,\eta_0)\big)$ and both
   vanish to infinite order on $\calG_{\bd}\times [0,\eta_0)$.
To see this note that in both \eqref{eq:kernel_r=0} and
\eqref{eq:lemma_kernel_expression} the leading order coefficient at
$\calG_3\times [0,\eta_0)$ does not depend on the  connection and hence
cancels 
upon taking the difference $K_{\hn}-K_{\on}$ (as long as $K_{\hn}, $ $K_{\on}$ are computed using the same density $\nu$). 
Finally, if $\eta=0$ we have $\td{E}_{0,\s}={E}_{0,\s}$, acting on
functions supported in a subset of $\oX=\oX_0\subset M_e^c$. 
Since $\on=\hn$ on $M_e^c$ and by the construction of
$\widehat{A}_{\chi_0,0,\s}$ 
(resp. $\o{A}_{\chi_0,0,\s}$), $K_{\hn}(\cdot, 0)$
(resp. $K_{\on}(\cdot,0)$) only depends on the connection $\hn$
(resp. $\on$) on $\oX\subset M_e^c$, 
 we have $\widehat{A}_{\chi_0,0,\s}=\o{A}_{\chi_0,0,\s}$ and thus $E_{0,\s}=\td{E}_{0,\s}=0$ and $W K_E(\cdot,0),$ $\,x_3^{-1}
 K_E(\cdot,0)\equiv0$.
\end{proof}

We finally have:
  
\begin{proof}[Proof of Proposition \ref{prop:mapping_diff_primed}.]

Recall that we now write $\oX$ for $\oX_0$.   
Let ${O}'\subset \oX$ be  a small open neighborhood of $p\in \oM_e$ in
$\oX$ where the results of this section hold and $O$ a neighborhood of $p$
in $\oX$ with $ O\subset K\subset O'$, where $K$ is compact.  
For sufficiently small $\eta\geq 0$ we have $\tU_\eta=\psi_\eta(U_\eta) \subset O$.
Fix $\d>0$. We will show that there exists an $\eta_0$ such that if $0\leq \eta< \eta_0$ then for  $u , v\in 
L^2 (O)$ with $\supp v\subset \tU_\eta$, and
$\td{W}_{\! j}\in\{x^2\p_x,x\p_{y^1},\dots,x\p_{y^n}\}\subset \calV_{\Sc}(\o{X})$ one has  
\begin{equation}\label{eq:small_delta_estimate}
 		|(u, \td{W}_j^k\,\td{E}_{\eta,\s}v)|
 		\leq \d\|u\|_{L^2(O)}\|v\|_{L^2 (\td{U}_\eta)}
                ,\quad j=0,\dots, n, \quad k=0,1. %
 \end{equation}
This will imply the claim since $\td{W}_j$ span $\calV_{\Sc}(\oX)$ on $O'$.
Let $\pi_{L;{\D\displaystyle}}=\pi_L\circ \b_{{\D\displaystyle}}$,
$\pi_{R;{\D\displaystyle}}=\pi_R\circ \b_{{\D\displaystyle}}$, where
$\pi_L$, $\pi_R$ denote 
projection onto the left and right factor of $\o{X}^2$
respectively.
By the Cauchy-Schwartz inequality and using the notations of Lemma \ref{lm:kernel_difference},
\begin{align}
\big|\int_{O^2}(u\otimes v)\k_{\td{E}_{\eta,\s}}\big|^2\leq&  
\Big(\int_{O_{{\D}}^2}\big|(\pi_{L;{\D}}^* u) K_E(\cdot,\eta)(\pi_{R;{\D}}^* v)\big|\nu\Big)^2\\
   \leq & \int_{O_{{\D}}^2}|(\pi_{L;{\D}}^* u)|^2 |K_E(\cdot,\eta)|\nu\cdot
  \int_{ O_{{\D}}^2}|K_E(\cdot,\eta)|\: |(\pi_{R;{\D}}^* v)|^2\nu.\quad\label{eq:cauchy_estimate}
 \end{align}

Recall that the ``coordinates'' \eqref{eq:polar_coordinates_near_front_face}
and the analogous ones given
by \begin{equation}\label{eq:coordinates_near_front_face_from_right} 
  \Big(\td{x},\td{y},\td{R}=\sqrt{\td{X}^2+|\td{Y}|^2},{\td{\th}}=(\td{X},\td{Y})/\td{R}\Big),\text{ where
  }\td{X}=\frac{x-\td{x}}{\td{x}\,^2}, 
  \td{Y}=\frac{y-\td{y}}{\td{x}} 
\end{equation}
identify $O_{{\D}}^2\setminus\calG_{\bd}^\cup$ with a subset of $\o{\R^{n+1}_+}\times [0,\infty)\times \S^{n}$.
By interchanging the roles of $(x,y)$ and $(\td{x},\td{y})$, Lemma
\ref{lm:smooth_density} yields the existence of a non-vanishing $\td{\a}\in
C^\infty (\oX_{{\D}}^2)$ such that in terms of
\eqref{eq:coordinates_near_front_face_from_right} one has
$\nu=\td{\a} (\td{R}+1)^{-1}(2+\td{x}\td{R}\td{{\th}}_0)^{n}|d\td{x}\,
d\td{y}\, d\td{R}\,d\td{{\omega}}|$ ($d\td{\omega}$ is the volume form with respect to the round metric). 
Thus
\begin{align}
  \int_{ O_{{\D}}^2}|(\pi_{L;{\D}}^* &u)|^2
  |K_E|\nu
  =\int
  |u(x,y)|^2
  |K_{E;L}(x,y,R,{\th},\eta)| \frac{(2+xR{\th}_0)^{n}}{1+R}\:|dx\:dy\:dR\:d{\omega}|,\qquad 
  \label{eq:left_integral}
\end{align}
and similarly
\begin{align}
  \int_{ O_{{\D}}^2}|K_{E}|& |(\pi_{R;\D}^* v)|^2\nu
  =\int
  |K_{E;R}(\td{x},\td{y},\td{R},{\td{{\th}}},\eta)|\:|v(\td{x},\td{y})|^2\frac{(2+\td{x}\td{R}\td{{\th}}_0)^{n}}{1+\td{R}}  
  |d\td{x}\:d\td{y}\:d\td{R}\:d{\td{{\omega}}}|,\qquad\label{eq:right_integral} 
\end{align}
where  $K_{E;L}$, $K_{E;R}$ express $K_E$ in terms of
\eqref{eq:polar_coordinates_near_front_face} and
\eqref{eq:coordinates_near_front_face_from_right} respectively.
The integrations on the right hand sides of \eqref{eq:left_integral} and
\eqref{eq:right_integral} are over the appropriate subsets of
$\o{\R^{n+1}_+}\times [0,\infty)\times \S^{n}$ corresponding to $O_{\D}^2$
  (the 
function $\ta$ and the corresponding function $\alpha$ have been absorbed
into $K_{E;R}$, $K_{E;L}$).  
Extend $K_{E;L}$ and $K_{E;R}$ to $\o{\R^{n+1}_+}\times [0,\infty)\times
  \S^{n}\times [0,\eta_0)$ by multiplication by a cutoff function in
$C^{\infty}_c(O'^2_{\D})$ which is $1$ in a neighborhood of $O_{\D}^2$.

For large $R_0$ we have 
\begin{align}
	\int|u(x,&y)|^2 |K_{E;L}(x,y,R,{\th},\eta)|\frac{(2+xR\th_0)^{n}}{1+R}|dx\:dy\:dR\:d{\omega}|\\*
	&\leq \|u\|^2_{L^2(O)}\sup_{(x,y)\in
          O}\int_{\S^n}\int_0^\infty|K_{E;L}(x,y,R,{\th},\eta)|\frac{(2+xR\th_0)^{n}}{1+R}|dR\:
        d{\omega}|\\ 
	&= \|u\|^2_{L^2(O)}\sup_{(x,y)\in
          O}\biggl(\int_{\S^n}\int_0^{R_0}|K_{E;L}(x,y,R,{\th},\eta)|\frac{(2+xR\th_0)^{n}}{1+R}|dR\:
        d{\omega}|\\* 
	&\quad +\int_{\S^n}\int_{R_0}^\infty
        (1+R)^{-2}\Big\{(1+R)|K_{E;L}(x,y,R,{\th},\eta)|{(2+xR\th_0)^{n}}\Big\}|dR\:
        d{\omega}|\biggr)\\ 
	&=\|u\|^2_{L^2(O)}\sup_{(x,y)\in O}\left(\I(x,y,\eta)+\II(x,y,\eta)\right).
\end{align}
By \eqref{eq:defining_functions}, ${(2+xR{\th}_0)^{n}}$ and $({1+R})$  are
of the {form} $x_{10}^{-n}$ and $x_{11}^{-1} x_{10}^{-2}$ 
respectively.
Since by Lemma \ref{lm:kernel_difference} $K_E$ vanishes to infinite order at
$\calG_{\bd}\times [0,\eta_0)$, there exists a constant $C$
such that for all $(x,y)\in O$ and all $0\leq\eta\leq
\eta_0$ $$(1+R)|K_{E;L}(x,y,R,{\th},\eta)|{(2+xR\theta_0)^{n}}\leq
C.$$ 
Therefore, for given $\d>0$, $R_0$  can be chosen sufficiently large that
$\II(x,y,\eta)\leq \d/2$ for $0\leq \eta\leq \eta_0$. 
On the other hand, $\I(x,y,\eta)$ is continuous (it is an integral over a
compact set of a function continuous jointly in $(x,y,R,{\th},\eta)$) and
it vanishes identically for $(x,y,\eta)\in
O\times\{0\}\subset K\times \{0\}$  by Lemma
\ref{lm:kernel_difference}. 
Thus there exists $\eta_0$ such that for $0\leq\eta\leq\eta_0 $ we have
$\sup_{(x,y)\in O}\I(x,y,\eta)\leq  \d/2$ and \eqref{eq:left_integral} is bounded above by $\d \|u\|^2_{L^2(O)}$.

Now \eqref{eq:right_integral} can be analyzed in exactly the same way as
\eqref{eq:left_integral}; the only difference is that now
${(2+\td{x}\td{R}\td{{\th}}_0)^{n}}$ and $({1+\td{R}})$ are of the form
$x_{01}^{-n}$ and $x_{11}^{-1} x_{01}^{-2}$.  
This however does not change the arguments since $ K_E$ vanishes to
infinite order at $\calG_{\bd}$, uniformly for small $\eta$. 
We conclude that  \eqref{eq:small_delta_estimate} holds for $k=0$.

To show \eqref{eq:small_delta_estimate} for $k=1$, we observe 
that
  $\b_{{\D}}^*(|dzd\td{z}|)=h\nu,$ where $h=x_{11}x_2^{n+2}x_{3}^n$
 (as before, $x_*$ stands for a boundary defining function of $\calG_*$
that is smooth and non-vanishing up to the other faces). 
By the analysis at the end of Section \ref{sec:scattering_calculus} it
follows that for $j=0,\dots,n$ the vector field $x_3 W_j$, where $W_j$ is
the lift of $\td{W}_j$, is smooth on $\oX_{\D}^2$ and tangent to all of its
boundary hypersurfaces. 
Thus $(W_jh)/h \in x_3^{-1} C^\infty(\oX_{{\D}}^2)$.
Writing $\k_{\td{E}_{\eta,\s}}=\td{\k}_E(z,\td{z},\eta)|dzd\td{z}|$ so that 
$\b_{{\D}}^*(\td{\k}_E)h=K_E$ we have, for  $u,$ $v\in L^2(O)$ as before, 
\begin{align}
	\int_{O^2} u(z)(\td{W}_{\!j}\,\,\td{\k}_E(z,\td{z},\eta))v(\td{z})|dzd\td{z}|&=
	\int_{O_{{\D}}^2} (\pi_{L;\D}^*u)\,\b_{{\D}}^* (\td{W}_{\! j}\, \td{\k}_E)(\pi_{R;\D}^*v)\, h\,\nu\\
	& =\int_{O_{{\D}}^2} (\pi_{L;\D}^*u)\,\big({W}_{\! j}\b_{{\D}}^* ( \td{\k}_E)\big )\,h\,(\pi_{R;\D}^*v)\,\nu\\
	& =\int_{O_{{\D}}^2} (\pi_{L;\D}^*u)\,\Big(W_j \,K_E-K_E \frac{W_{\!j}\,h}{h}\Big) (\pi_{R;\D}^*v)\, \nu.
\end{align}
Then \eqref{eq:small_delta_estimate} for $k=1$ follows exactly the same
steps  as for $k=0$ from \eqref{eq:cauchy_estimate} onwards, with $K_E$
replaced by $W_j \,K_E-((W_{\!j}\,h)/h)K_E $: 
by Lemma \ref{lm:kernel_difference}, $W_j \,K_E-((W_{\!j}\,h)/h)K_E \in
C^0\big(O_{\D}^2\times[0,\eta_0)\big)$, it vanishes to infinite order at
$\calG_{\bd}\times [0,\eta_0)$ and is identically 0 for $\eta=0$. 
This finishes the proof of the proposition.  
\end{proof}

\noindent \textbf{Acknowledgments.}
Research of N.E. was partially supported by the National Science Foundation
under Grant No. DMS-1800453 of Gunther Uhlmann. 
The authors would like to thank Hart Smith, Gunther Uhlmann, and Andr\'as
Vasy for helpful discussions.
This paper is based on Chapter 1 of N.E.'s University of Washington PhD
Thesis (\cite{EptaminitakisNikolaos2020Gxto}).  

\bigskip

\noindent\textit{The second author fondly remembers the time he spent in the company of
Vaughan Jones at the 2008 Summer Workshop of the New Zealand Mathematics
Research Institute in Nelson. Vaughan's support and presence were felt
throughout the week, from his perspicacious comments and questions during
the lectures to his enthusiasm for extracurricular beach and water
activities to after hours socializing. He enriched the mathematics and
the lives of those who had the good fortune to be around him.}

\bigskip

\bibliographystyle{abbrvalpha}

\bibliography{mybib.bib}

\end{document}